\documentclass[a4paper,12pt]{article}
\usepackage[a4paper,margin=1in,headheight=\baselineskip]{geometry}
\usepackage[utf8]{inputenc}
\usepackage[T1]{fontenc}
\usepackage[english]{babel}
\usepackage[hidelinks,hypertexnames=false]{hyperref}
\usepackage{graphicx}
\usepackage{amsmath}
\usepackage{amssymb}
\usepackage{amsthm}
\usepackage{cleveref}
\usepackage{thmtools}
\usepackage{mathtools}
\usepackage{bbm}
\usepackage{enumerate}
\usepackage{enumitem}
\usepackage{color}
\numberwithin{equation}{section}

\usepackage{scrlayer-scrpage}
\ofoot{}

\crefname{subsection}{subsection}{subsections}
\crefname{subsubsection}{subsubsection}{subsubsections}

\usepackage{tikz}
\usetikzlibrary{arrows,positioning}

\let\originalleft\left
\let\originalright\right
\renewcommand{\left}{\mathopen{}\mathclose\bgroup\originalleft}
\renewcommand{\right}{\aftergroup\egroup\originalright}

\newcommand{\R}{{\mathbb R}}

\newcommand{\C}{{\mathbb C}}

\newcommand{\T}{\mathcal{T}}

\newcommand{\V}{\mathcal{V}}
\newcommand{\W}{\mathcal{W}}

\renewcommand{\d}{\;\mathrm{d}}

\newcommand{\GammaD}{\Gamma_{\mathrm{D}}}
\newcommand{\GammaN}{\Gamma_{\mathrm{N}}}

\newcommand{\B}{X}
\newcommand{\Bd}{B}
\newcommand{\Bdh}{B_\diamP}
\newcommand{\Vh}{\V_\diamT}

\DeclareMathOperator{\Div}{div}

\DeclareMathOperator{\opL}{\mathcal L}
\DeclareMathOperator{\opl}{\ell}
\DeclareMathOperator{\opN}{\mathcal N}
\DeclareMathOperator{\opSL}{\mathcal{S}}
\DeclareMathOperator{\opDL}{\mathcal{D}}
\DeclareMathOperator{\opV}{\mathsf{V}}
\DeclareMathOperator{\opVh}{\opV_{\mathgroup=-1 \diamT}}
\DeclareMathOperator{\opVht}{\widetilde{\opV}_{\mathgroup=-1 \diamT}}

\newcommand{\opA}{\mathbb A}
\newcommand{\Nh}{\widetilde{N}_\diamT}

\DeclareMathOperator{\trace}{\gamma}
\newcommand{\traceD}[1][]{\trace_{#1}}

\newcommand{\traceS}{\trace_{\Sigma}}

\newcommand{\conj}{\overline}

\newcommand{\genvarV}{v}
\newcommand{\genvarValt}{w}

\newcommand{\genvarB}{t}
\newcommand{\datavarB}{f}
\newcommand{\genvarBd}{g}
\newcommand{\genvarBdalt}{h}
\newcommand{\solvarBd}{b}

\newcommand{\notempty}[2][]{\expandafter\ifx\expandafter\relax
\detokenize{#2}\relax\else#1{(#2)}\fi}
\newcommand{\sscript}[1]{^{#1}}
\newcommand{\genvarVh}[1][]{v_{\diamT}\notempty[\sscript]{#1}}

\newcommand{\genvarBdh}[1][]{g_\diamP\notempty[\sscript]{#1}}
\newcommand{\genvarBdhalt}{h_\diamP}
\newcommand{\solvarBdh}{b_\diamP}

\newcommand{\NVh}{n_{\diamT}}
\newcommand{\NBdh}{n_{\diamP}}

\newcommand{\diamT}{d}
\newcommand{\diamP}{h}
\newcommand{\errCon}{\varepsilon_{\mathrm{con}}}

\newcommand{\traceN}[1][N]{\operatorname{\gamma}_{\mathrm{N},#1}}

\newcommand{\jump}[2][]{\left[#2\right]_{#1}}

\newcommand{\mean}[2][]{\{\kern-4.2pt\{#2\}\kern-4.2pt\}_{#1}}

\newcommand{\jumpN}[2][]{\jump[\mathrm{N}\ifstrempty{}{}{,#1}]{#2}}

\theoremstyle{definition}
\newtheorem{defn}{Definition}[section]
\theoremstyle{plain}
\newtheorem{theorem}{Theorem}[section]
\newtheorem{lemma}[theorem]{Lemma}
\newtheorem{remark}[theorem]{Remark}

\newtheorem{cor}[theorem]{Corollary}

\newcommand{\Const}[1]{C_{\mathrm{#1}}}
\newcommand{\const}[1]{c_{\mathrm{#1}}}

\usepackage{totcount}
\newtotcounter{nconst}
\makeatletter
\newcount\rc@count
\let\rc@clearconstantlist\empty

\newcommand\rc@clearconstant[1]{\global\expandafter\let\csname rc@XConst@#1\endcsname\undefined}

\newcommand\resetconstants[1]{%
    \def\rc@constname{#1}
    \global\rc@count=1\relax 
    \bgroup 
        \let\\\rc@clearconstant 
        \rc@clearconstantlist
        \global\let\rc@clearconstantlist\empty 
    \egroup
}
\newcommand\XConst[1]{%
    \@ifundefined{rc@XConst@#1}{%
        \expandafter\xdef\csname rc@XConst@#1\endcsname{%
           \noexpand\rc@useconst{\rc@constname}{\the\rc@count}%
        }%
        \g@addto@macro\rc@clearconstantlist{\\{#1}}%
		\global\setcounter{nconst}{\the\rc@count}\relax
        \global\advance\rc@count1\relax
    }{}%
    \csname rc@XConst@#1\endcsname
}

\newcommand\rc@useconst[2]{\ensuremath{#1_{#2}}}
\makeatother

	\resetconstants{C}

\title{ BEM for variable coefficient second-order problems
}
\author{Benedikt Gr\"a{\ss}le\footnote{Institut für Mathematik, Universität Zürich, Winterthurerstr.~190, CH-8057 Zürich, Switzerland \\
	(\{benedikt.graessle, stas\}@math.uzh.ch).}
	\and
	Stefan A. Sauter\footnotemark[1]
}
\date{}
\begin{document}

\maketitle
\begin{abstract}
	A novel boundary element method (BEM) removes the classical dependence 
	on explicit fundamental solutions and extends quasi-optimal BEM 
	discretisations to strongly elliptic operators with variable coefficients. 
	The approach constructs a computable approximation of the boundary operator 
	from a Galerkin discretisation of the underlying elliptic differential operator 
	in a one-time preprocessing step, for instance by conforming finite elements. 
	The resulting algebraic formulation retains the dimension reduction intrinsic 
	to boundary integral methods and is compatible with standard data-sparse 
	matrix compression techniques.

\end{abstract}

\noindent {\bf Keywords:} boundary element method, boundary integral equation,
variable-coefficient elliptic operator, 
single-layer potential, quasi-best approximation

\noindent {\bf AMS Classification:}  65N38, 65N30, 65R20, 31B10

\section{Introduction}\label{sec:Introduction}
Boundary element methods (BEM) provide efficient discretisations of
boundary integral equations associated with linear elliptic partial differential
equations (PDE). In its classical form, the method of integral equations
exploits an explicit Green's function/fundamental solution for the underlying
differential operator and represents boundary integral operators via their singular kernels.
This leads to dimension reduction, accurate treatment of unbounded domains,
and high accuracy for smooth data. However, the reliance on explicit kernels
is a structural limitation: For general second-order operators with variable
coefficients, closed-form representations are typically \emph{not} available.
While the efficient numerical realisation of BEM for problems with
\textit{constant} coefficients has matured (see, e.g.,%
~\cite{SS:BoundaryElementMethods2011}), the discretisation of boundary integral operators
whose kernel functions are given by oscillatory integrals or asymptotic
expansions requires the development of problem-specific quadrature methods for
an accurate and efficient evaluation of these kernels. We emphasize that for
most boundary value problems with varying coefficients, the Schwartz kernel is
unknown and an evaluation not feasible.

This paper introduces a kernel-free boundary element formulation that (i) replaces the
\textit{analytic derivation} of the boundary integral operator for 
general strongly elliptic second-order PDE by a
\emph{computable approximation} obtained from a volume discretisation, and (ii) 
constructs the discrete boundary operator directly from the discretised 
PDE operator, using only information from a local neighbourhood of the boundary or interface.
In this way, the principal benefit
of BEM---a reduction of the effective problem dimension by one---is retained,
in particular for transmission problems with moving interfaces 
or with interfaces at various positions, embedded in the same ambient medium.
From an algorithmic viewpoint, the resulting algebraic structure is
compatible with data-sparse realisations in hierarchical matrix formats.

For the presentation, a single-layer potential is considered and the resulting
approximate single-layer boundary integral operator is employed
in a Galerkin boundary element formulation. First, the volume solution operator 
(also called Newton potential) is discretised from a 
conforming Galerkin finite element method (FEM) for the underlying elliptic PDE.
This volume computation is performed only once in a preprocessing step,
independent of the boundary discretisation, and depends only on the coefficients 
of the differential operator. The preprocessing cost is therefore 
comparable to an (approximate) inversion of 
a FEM stiffness matrix and does
\emph{not} need to be repeated when interfaces change within the same medium.

Second, the discrete single-layer operator is defined purely algebraically as a composition of
the inverse finite element stiffness matrix with sparse transfer operators
between the discrete volume and boundary spaces. 
This factorised representation enables a direct implementation 
once the volume discretisation is available
and keeps the additional cost of the boundary discretisation~moderate.
Importantly, the factorisation is compatible with hierarchical matrix
compression, which enables the application of the resulting
discrete boundary operators in (almost) linear complexity up to logarithmic
factors.

This new approach constitutes the main contribution. It offers a kernel-free,
black-box route to nonlocal boundary and interface operators for elliptic
problems, without restricting to constant coefficients and without requiring
problem-specific kernel evaluations (for instance, via oscillatory-integral
representations or asymptotic expansions).

The analysis is developed in an abstract operator framework and then
specialised to general second-order strongly elliptic operators. This
separates the functional-analytic properties of the single-layer operator from
the approximation properties of the chosen volume and boundary
discretisations.\medskip

\renewcommand*{\thetheorem}{\Alph{theorem}}
\medskip
\noindent\textbf{Main results.} 
Let $\Omega\subset\R^n$ be a Lipschitz domain and consider
the elliptic operator 
\begin{align*}
	\opL v\coloneqq -\Div(\opA\nabla v) + \mathbf{b}\cdot \nabla v + c\, v
	\qquad\text{for all }v\in \V\subset H^1(\Omega)
\end{align*}
with (possibly) varying coefficients $\opA,\mathbf{b}, c$, and boundary conditions
on $\partial\Omega$.
The focus is on $\mathcal{L}%
$-harmonic functions in $\Omega\setminus\Sigma$ outside some Lipschitz
manifold $\Sigma\subset\overline{\Omega}$. Let $B:=H^{-1/2}\left(
\Sigma\right)  $. 
A mapping $\mathcal{S}:B\rightarrow\mathcal{V}$ is called 
\textit{single-layer} \textit{potential} if, for any $b\in H^{-1/2}\left(
\Sigma\right)  $, 
\[
u=\mathcal{S}b\in\mathcal{V\quad}\text{satisfies\quad}\mathcal{L}%
u=0\quad\text{in }\Omega\setminus\Sigma.
\]
Its trace %
on $\Sigma$
induces the \emph{single-layer equation}
\begin{align}\label{eqn:WP_intro}
	\opV\solvarBd=\datavarB
	\qquad\text{on }\Sigma
\end{align}
with prescribed Dirichlet data 
$\datavarB\coloneqq u|_{\Sigma}\in \B\coloneqq H^{1/2}(\Sigma)$.
The integral kernel of the single-layer (boundary) operator $\opV:\Bd\to \B$ 
is the Green's function associated to $\opL$ and turns~\eqref{eqn:WP_intro} into
a boundary integral equation of the first kind with boundary density 
$\solvarBd$ as unknown.
Classical Galerkin BEM for~\eqref{eqn:WP_intro} employs 
some finite-dimensional conforming subspace $\Bdh\subset \Bd$ and an explicit 
representations of~$\opV$, which is not available in general.

\medskip\noindent
\emph{Kernel-free Boundary Element Method.}
To accomodate varying coefficients, the key idea is a construction of the discrete 
single-layer operator based on the characterisation of the single-layer potential 
as the composition of the dual trace operator 
with the solution operator (Newton potential) for $\mathcal{L}$. 
First, we introduce a discretisation of the latter over a
finite-dimensional conforming subspace $\mathcal{V}_{d}\subset\mathcal{V}$.
The approximated single-layer operator
\[
\mathsf{V}_{d}:B\rightarrow X
\]
is then defined as the composition with discrete trace and dual
trace maps. This leads to a generalized Galerkin BEM formulation for
(\ref{eqn:WP_intro}): Seek $b_{h}\in B_{h}\subset B$ with
\begin{equation}
\left\langle \mathsf{V}_{d}b_{h},\overline{g_{h}}\right\rangle _{X\times
B}=\left\langle f,\overline{g_{h}}\right\rangle _{X\times B}\qquad\text{for
all }g_{h}\in B_{h}.\label{eqn:DSP_intro}%
\end{equation}
The passage from a homogeneous PDE to a boundary integral equation reduces the
effective dimension of the discretisation by one. This benefit extends the kernel-free 
setting with variable coefficients. 
Importantly, the fully discrete method~\eqref{eqn:DSP_intro} admits a purely algebraic formulation
and its solution is therefore computable.

\medskip\noindent\emph{Analysis.}
For the mathematical analysis of the presented method, it is convenient to 
formulate the discrete maps at the operator level in an abstract setting.
This paper considers abstract discretisations $\mathcal{V}_{d}\subset \mathcal{V}$ 
and $B_{h}\subset\Bd$
to isolate the structural mechanism from
the particular choice of discrete volume and boundary spaces.
The first main result establishes a quasi-approximation property
with respect to both discretisations, under the assumption that the consistency error
$\varepsilon_{\mathrm{con}}\coloneqq\|\opV-\opVh\|_{L(\Bdh;\Bdh')}$ is sufficiently small.%
\begin{theorem}\label{thm:A}
	If~$\varepsilon_{\mathrm{con}}\leq\varepsilon_{\opL}$, then~\eqref{eqn:DSP_intro} 
	admits a unique discrete solution $\solvarBdh\in \Bdh$ with
	\begin{align*}
		\Const{qb}^{-1}\|\solvarBd-\solvarBdh\|_{H^{-1/2}(\Sigma)}
		\leq \min_{\genvarBdh\in \Bdh}\|\solvarBd-\genvarBdh\|_{H^{-1/2}(\Sigma)}
			+  \min_{\genvarVh\in \Vh}\|\opSL\solvarBd-\genvarVh\|_{H^1(\Omega)}.
	\end{align*}
\end{theorem}
The point is that the constants $\varepsilon_{\opL},\Const{qb}>0$ in \textbf{\Cref{thm:A}}
depend exclusively on the elliptic differential operator $\opL$ and are,
in particular, 
independent of the interface $\Sigma$ and the source $f\in H^{1/2}(\Sigma)$.
The convergence analysis is discussed for Galerkin FEM and BEM
discretisations with (mapped) piecewise polynomials $\Vh=P_p(\mathcal{T})\cap \V$ and
$\Bdh=P_k(\mathcal{P})$ of total degrees at most $p$ and $k$
over quasi-uniform
partitions $\mathcal{T}$ and $\mathcal{P}$ of $\Omega$ and $\Sigma$, respectively.
In this setting, the smallness condition $\varepsilon_{\mathrm{con}}\leq\varepsilon_{\opL}$
translates to the natural compatibility condition 
\begin{align}\label{eqn:rel_mesh_size}
	\diamT\leq \Const{cmp} \diamP
	\qquad\text{with }\diamT=\max_{T\in\mathcal{T}}\operatorname{diam}(T)
	\text{ and }\diamP=\max_{\tau\in\mathcal{P}}\operatorname{diam}(\tau)
\end{align}
on the maximal mesh sizes with some constant $\Const{cmp}>0$
determined by $\opL$ and an index elliptic regularity $\sigma_{\operatorname{reg}}>0$ 
defined in~\Cref{sub:Abstract_model_problem}.
\textbf{\Cref{thm:A}} and well-known approximation properties of finite/boundary
elements imply optimal convergence rates.

\begin{theorem}\label{thm:B}
	There exists $\Const{cmp}>0$ such that~\eqref{eqn:rel_mesh_size}
	implies the well-posedness of~\eqref{eqn:DSP_intro}.
	Moreover, 
	for any $0\leq s\leq \sigma_{\mathrm{reg}}$ and $\datavarB\in H^{1/2+s}(\Sigma)$
	it holds 
	(i) if $s<1/2$, then
	\begin{align*}
		\|\solvarBd-\solvarBdh\|_{H^{-1/2}(\Sigma)}
		\leq \Const{rate} (\diamT^{s}+\diamP^s)\|\datavarB\|_{H^{1/2+s}(\Sigma)},
	\end{align*}
	(ii) if $s>1/2$ and $\Sigma\subset\cup\mathcal{F}$ is resolved by 
	the faces $\mathcal{F}$ of $\mathcal{T}$, then
	\begin{align*}
		\|\solvarBd-\solvarBdh\|_{H^{-1/2}(\Sigma)}
		\leq \Const{rate}\left(\diamT^{\min\{s, p+1\}} + \diamP^{\min\{s,k+1\}}\right)
			\|\datavarB\|_{H^{1/2+s}(\Sigma)}.
	\end{align*}
\end{theorem}
For merely Lipschitz coefficients and boundaries, 
one has $\sigma_{\mathrm{reg}}<1/2$ and optimal 
a~priori convergence rates by \textbf{\Cref{thm:B}}.i. 
Higher-order convergence rates (as in \textbf{\Cref{thm:B}}.ii) 
require additional regularity of the data together with 
an appropriate resolution of the interface $\Sigma$ by the volume discretisation
that is further discussed in \Cref{rem:delta_resolve_interface}.

\newpage
\medskip\noindent
\emph{Complexity.} 
The (approximated) single-layer operator $\opV_\diamT$ in~\eqref{eqn:DSP_intro} is nonlocal 
on an $\left(  n-1\right)  $-dimensional manifold $\Sigma$ and results in a dense system matrix
$V_\diamT^\diamP\in \mathbb{C}^{\NBdh\times\NBdh}$, so efficient algorithms are essential.
The first key observation to this end is the factorisation
\begin{align}\label{eqn:Vd_repr}
	V_{\diamT}^{\diamP}=(R_{\diamT}^{\diamP})^{H}N_{\diamT}R_{\diamT}^{\diamP}
\end{align}
into a sparse transfer matrix \(R_{\diamT}^{\diamP}\)
between the volume and boundary spaces, and the inverse
volume stiffness matrix \(N_{\diamT}\).
The point is that \(R_{\diamT}^{\diamP}\) is computable (subject to an exact quadrature)
in linear time with respect to the boundary dimension $\NBdh=\operatorname{dim}\Bdh$
and the construction of \(V_{\diamT}^{\diamP}\) accesses
\(N_{\diamT}\) only through its action on boundary-adjacent volume degrees of freedom.
This enables an efficient assembly of the system matrix with the same computational complexity 
as for the dense Galerkin matrix in classical BEM.
\begin{theorem}\label{thm:C}
	If $N_{\diamT}$ is available for degrees of freedom in a vicinity of the boundary and
	given $R_\diamT^\diamP$ for $\diamP\lesssim\diamT$, 
	then the system matrix $V_{\diamT}^{\diamP}$ can be computed with
	$\mathcal{O}\left( \NBdh^{2}\right) $ operations.
\end{theorem}

In practice, fast BEM implementations avoid forming dense matrices by data-sparse
compression such as hierarchical matrix approximations or multipole techniques. 
The factorisation~\eqref{eqn:Vd_repr} is compatible with this paradigm, since one may replace $N_\diamT$
by a data-sparse approximation; see \Cref{rem:data_sparse}.
Such approximations of inverse FEM stiffness matrices are available, e.g., in hierarchical format as in%
~\cite[\S 7.5]{Hac:HierarchicalMatricesAlgorithms2015} and~\cite{FMP:HmatrixApproximability2015}.

\medskip
\noindent\textbf{Outline and further results.}
\Cref{sec:preliminaries} introduces the relevant notation and definitions.

The quasi-optimal discretisation of the single-layer operator is presented 
in an abstract Hilbert space setting 
with conforming subspaces $\Vh\subset \V$ and $\Bdh\subset \Bd$ in 
\Cref{sec:BEM_for_the_single-layer_potential}.
The differential and boundary integral operators are characterised at the operator level and,
in particular, without drawing on kernel representations to underline the kernel-free nature and scope 
of the theory. 
This framework enables BEM formulations for a large problem class by identifying 
minimal conditions for quasi-optimal approximation properties of the discrete spaces
in \Cref{thm:quasi_best}. 
A key thechnical result is the
control of the consistency error $\varepsilon_{\mathrm{con}}$
in terms of discretisation parameters as in~\eqref{eqn:rel_mesh_size}
in \Cref{lem:consistency}.

\Cref{sec:BEM for general single-layer operators} applies the abstract theory to
single-layer operators induced by strongly elliptic second-order operators with general
coefficients, leading to the first quasi-optimal BEM formulation in this setting and
implying \textbf{\Cref{thm:A}} in \Cref{sub:Model problem}.
The analysis accommodates a broad class of homogeneous boundary conditions, including
transparent (Dirichlet-to-Neumann) conditions used to treat (possibly indefinite) problems 
in unbounded domains, see
e.g.,~\cite{MM:FiniteElementMethod1980,Fen:FiniteElementMethod1983,KG:ExactNonreflectingBoundary1989,GS:DirichlettoNeumannOperatorHelmholtz2025,GHS:StableSkeletonIntegral2025}
and the references therein.
\Cref{sub:Elliptic regularity} discusses the elliptic regularity of the underlying PDE and
transmission problem that enter the compatibility condition~\eqref{eqn:rel_mesh_size},
using results from standard
theory~\cite{Gri:EllipticProblemsNonsmooth1985,Sav:RegularityResultsElliptic1998,McL:StronglyEllipticSystems2000,Agr:SpectralProblemsSecondorder2002}%
. 
\Cref{sub:Relation_to_classical_BEM} relates the operator-oriented framework to 
classical kernel-based representations via Green's functions.

\Cref{sub:FEM-based discretiation} instantiates the abstract setting for
standard $h$-FEM/BEM discretisations based on fairly general mapped piecewise polynomials on
shape-regular, quasi-uniform partitions of $\Omega$ and $\Sigma$.
The construction allows for non-matching volume and boundary discretisations;
the interface $\Sigma\not\subset\cup\mathcal{F}$ need \emph{not} be resolved by the volume mesh.
\Cref{sub:Convergence rates in } recalls required approximation properties and inverse
estimates, and derives the convergence rates of \textbf{\Cref{thm:B}}
under the natural compatibility condition~\eqref{eqn:rel_mesh_size}.
\Cref{rem:delta_resolve_interface} discusses higher-order convergence rates as in \textbf{\Cref{thm:B}}.ii 
for nearly resolved interfaces from~\cite{LMWZ:OptimalPrioriEstimates2010}.
The sparsity of the transfer matrix and~\eqref{eqn:Vd_repr}
lead in \Cref{sub:Algebraic structure and sparsity} to the complexity control \textbf{\Cref{thm:C}}, 
restated as \Cref{thm:complexity}, under a standard complexity model for sparse-dense
matrix multiplication.
\setcounter{section}{1}
\renewcommand*{\thetheorem}{\arabic{section}.\arabic{theorem}}
\section{Preliminaries}%
\label{sec:preliminaries}

Complex conjugation of $z\in\mathbb{C}$ is denoted by $\conj{z}$
and the conjugate transpose of a complex matrix $M$ by $M^H$;
the space of Hermitian $n\times n$-matrices reads $\mathbb{H}^n\subset \mathbb{C}^{n\times n}$.

For a complex-valued Banach space $X$ with norm $\|\bullet\|_X$, 
we denote its (topological) dual by
$X'$ and use the (bilinear) duality pairings $\langle\bullet,\bullet\rangle_{X'\times
X}$ and $\langle\bullet,\bullet\rangle_{X\times X'}$.
The space $L(X;Y)$ of bounded linear maps $A:X\to Y$ between Banach spaces $X$ and $Y$
is equipped with the operator norm $\|A\|_{L(X;Y)}$.
A conjugation $\conj{\bullet}:X\to X$ on $X$ is
an antilinear isometric involution and
extends to a conjugation $\conj{\bullet}:X'\to X'$ on $X'$ by
\begin{align*}
	\left\langle \conj {x'}, {x}\right\rangle_{X'\times X}=
	\conj{\left\langle x', \conj{x}\right\rangle_{X'\times X}}
	\qquad\text{for all }x\in X, x'\in X'.
\end{align*}
If $X$ and $Y$ carry a conjugation, the adjoint $A'\in L(Y';X')$ of $A\in L(X;Y)$ is given by
\begin{align}\label{eqn:adjoint_def}
  \big\langle A'y',\conj x\big\rangle_{X'\times X}
  =\big\langle y',\conj{Ax}\big\rangle_{Y'\times Y}
  \qquad\text{for all }x\in X,\ y'\in Y'.
\end{align}
Given a linear subspace $W\subset H$ of the Hilbert space $H$ with 
inner product $\left(\bullet,\bullet\right)_H$,
we write $v\perp W$ if $(v,w)_H=0$ for all $w\in W\subset H$.

Standard notation applies to (complex-valued) Lebesgue and Sobolev spaces and their norms
for open sets $\omega\subset\R^n$ with $n\geq2$ throughout this paper.
In particular, $L^2(\omega;X)$ denotes square-integrable functions
with values in $X\in\{\mathbb{C},\mathbb{C}^n,\mathbb{H}^n\}$.
We abbreviate $L^2(\omega):=L^2(\omega;\mathbb{C})$ and use the same convention for
Sobolev spaces. The spaces $H^s(\omega)$ for $s>0$ are
understood by interpolation; negative-order spaces are defined by duality 
with respect to the $L^2$-pairing and endowed with the operator norm.
Given $\opA\in L^\infty(\omega;\mathbb{H}^n)$,~set
\begin{align*}
  H^1(\omega,\opA):=\{v\in H^1(\omega):\ \Div(\opA\nabla v)\in L^2(\omega)\}.
\end{align*}
For the definition of Sobolev spaces $H^s(\Gamma)$ with $s\in\mathbb R$
on $(n-1)$-dimensional Lipschitz surfaces $\Gamma\subset\mathbb{R}^n$,
we refer to~\cite[pp.~96--99]{McL:StronglyEllipticSystems2000}.

Let $\omega\subset\R^n$ be a Lipschitz set; a Lipschitz domain is a connected Lipschitz set.
We call $\Gamma\subset\partial\omega$ an isolated boundary component of $\omega\subset\R^n$ 
if $\operatorname{dist}(\partial\omega\setminus\Gamma, \Gamma)>0$.
A Lipschitz interface $\Gamma\subset\overline\omega$ in $\omega\subset\R^n$
is an isolated boundary component $\Gamma\subset \partial G$ of some
(not explicitly referenced) Lipschitz set $G\subset\omega$.
The Dirichlet trace $\traceD[\Gamma]:H^1(\omega)\to H^{1/2}(\Gamma)$ associated 
with a Lipschitz interface $\Gamma\subset\overline \omega$ 
is a bounded linear surjection that extends the classical pointwise trace. 
Its restriction to relatively closed subsets $\Gamma_{\mathrm{D}}\subset \Gamma$~defines
\[
	H^1_{\Gamma_{\mathrm{D}}}(\omega)
	:=\{v\in H^1(\omega):\ \traceD[\Gamma]v=0\ \text{on }\Gamma_{\mathrm{D}}\}.
\]
The (primal) trace space $H^{1/2}(\Gamma)$ is equipped with the minimal extension norm
\begin{align*}
	\|\genvarB\|_{H^{1/2}(\Gamma)}
	\coloneqq 
	\inf_{\substack{\genvarV\in H^1(\omega)\\\traceD[\Gamma] \genvarV=\genvarB}}
	\|\genvarV\|_{H^1(\omega)}
	\qquad\text{for all }\genvarB\in H^{1/2}(\Gamma).
\end{align*}
This natural trace norm is equivalent to classical norms%
~\cite{LM:NonhomogeneousBoundaryValue1972,Gra:OptimalTraceNorms2025}.
For $\opA\in L^\infty(\omega;\mathbb{H}^n)$ clear from the context, the co-normal trace
$\traceN[\omega]:H^1(\omega,\opA)\to H^{-1/2}({\partial \omega})$
is a bounded surjection
onto the  dual space of $H^{1/2}({\partial \omega})$
defined by Green's identity
\begin{align*}
	\left\langle \traceN[\omega]v, \traceD[\partial\omega] w\right\rangle_{
			{\partial\omega}} 
			= \int_\omega \opA\nabla v\cdot \nabla w + \Div (\opA\nabla v)w \d x\qquad\text{for all }v\in
			H^1(\omega,\opA),w\in H^1(\omega)
\end{align*}
with the abbreviation $\left\langle \bullet, \bullet\right\rangle_{{\partial \omega}}$ 
for the duality pairing 
on $H^{-1/2}({\partial \omega})\times H^{1/2}({\partial \omega})$.

The context-sensitive notation $|\bullet|$ may refer to 
the cardinality $|\mathcal{I}|$ of a countable set $\mathcal{I}$ and the absolute value $|z|$
of $z\in\mathbb{C}$.
We abbreviate $[J]\coloneqq\{1,\dots,J\}$ for $J\in\mathbb{N}_0$ and 
write $a\lesssim b$ for $a\le C\,b$ with a constant $C>0$ 
independent of discretisation parameters; $a\approx b$ abbreviates
$a\lesssim b\lesssim a$.
For a matrix $M$, $\operatorname{nnz}(M)$ denotes the number of nonzero entries
and $\|M\|_2$ the spectral norm induced by the (complex)
Euclidean inner product~$\left\langle \bullet,\conj\bullet\right\rangle _2$.
\section{BEM for an abstract single-layer problem}%
\label{sec:BEM_for_the_single-layer_potential}
This section introduces an abstract boundary element discretisation for
the single-layer operator associated to general elliptic differential operators.

\subsection{Abstract model problem}%
\label{sub:Abstract_model_problem}
The linear differential operator is represented by a bounded operator
$\opL:\V\to \V'$ from some Hilbert space $\V$,
equipped with a norm $\|\bullet\|_{\V}$ and a conjugation (denoted by
$\conj{\bullet}$),
into its (linear) dual space $\V'$.
The associated sesquilinear form $\opl:\V\times \V\to \mathbb{C}$ reads for any $v,w\in \V$
as $\opl(v,w) \coloneqq \left\langle \opL v,\conj w\right\rangle_{\V'\times \V}$.
We require $\opL$ to be invertible and, to simplify the discussion,
even assume that $\opl$ is coercive:
There exists $\const{\opL}>0$ with
\begin{align}\label{eqn:L_coercive}
	|\opl(v,v)|
	\geq \const{\opL}\|v\|_{\V}^2
\qquad\text{for all }v\in \V.
\end{align}
Its inverse plays the role of the (generalised) Newton potential and is
denoted by $\opN\coloneqq \opL^{-1}:\V'\to \V$.
The abstract \emph{trace} 
$\traceD:\V\to \B$ is a bounded linear surjection onto the primal trace space $\B$
with $\V_0\coloneqq\ker\traceD\subset \V$ closed. 
Hence $\B$ naturally inherits a Hilbert space structure 
(topologically equivalent to the quotient space $\V/\V_0$)
with norm
\begin{align}\label{eqn:B_norm_def}
	\|\genvarB\|_{\B}\coloneqq\inf_{\substack{v\in \V\\\traceD v = \genvarB}}\|v\|_{\V}
	\qquad\text{for all }\genvarB\in \B.
\end{align}
In particular, $\V=\V_0\oplus \V_0^\perp$ with the orthogonal complement 
$\V_0^\perp\coloneqq\{v\in \V\ :\ v \perp \V_0\}$ and $\traceD:\V_0^\perp \to \B$
is a (by~\eqref{eqn:B_norm_def} isometric) isomorphism.
Hence $\|\traceD\|_{L(\V;\B)}=1$ and~\eqref{eqn:B_norm_def}
is the canonical trace norm associated with the (differential) operator $\opL:\V\to\V'$,
cf.~\cite{Gra:OptimalTraceNorms2025} for PDE-level interpretations and equivalences to standard trace norms.
We abbreviate 
$\|\opL\|\coloneqq \|\opL\|_{L(\V;\V')}$ and 
$\|\opN\|\coloneqq \|\opN\|_{L(\V';\V)}$, and $\|\traceD\|\coloneqq\|\traceD\|_{L(\V;\B)}$ in the following.

\medskip
Given $\datavarB\in\B$, the abstract model transmission problem seeks a solution $u\in \V$ to
\begin{subequations}\label{eqn:AMP}
	\begin{align}\label{eqn:AMPa}
		\opL u 
		&= 0\qquad\text{in }\V_0',\\
		\traceD u &=\datavarB\qquad\text{in }\B.\label{eqn:AMPb}
	\end{align}
\end{subequations}
A function $u\in \V$ with~\eqref{eqn:AMPa} is called $\opL$-harmonic in $\V_0'$,
where~\eqref{eqn:AMPa} is understood as 
\begin{align}\label{eqn:L-harmonic}
	\opl(u,w) = 0
	\qquad\text{for all }w\in \V_0\subset \V.
\end{align}
It is clear that~\eqref{eqn:AMP} is well posed %
and a natural question is that of an explicit representation of the unique solution $u\in \V$
in terms of its boundary data $\datavarB\in \B$.
To this end, introduce the dual trace space $\Bd\coloneq\B'$
and the single-layer potential 
\begin{align}\label{eqn:SL_def}
	\opSL\coloneqq\opN\traceD': \Bd\to \V
\end{align}
that becomes important in the following discussion.
Here $\traceD':\Bd\to \V'$ with abbreviated norm 
$\|\trace'\|\coloneqq\|\trace'\|_{L(\Bd;\V')}$
denotes the adjoint (defined by~\eqref{eqn:adjoint_def}) of the trace $\gamma$.

\begin{theorem}[single-layer potential]\label{thm:SL}
	The single-layer potential~\eqref{eqn:SL_def} is a bounded linear operator
	with norm 
	$\|\opSL\|\coloneqq\|\opSL\|_{L(\Bd;\V)}\leq \|\opN\|$ and, 
	for any $\genvarBd\in \Bd$, 
	$v\coloneqq\opSL\genvarBd\in \V$ is the unique solution to~\eqref{eqn:AMP} for
	$f\coloneqq \traceD v\in \B$.
\end{theorem}
\begin{proof}
	Observe $\|\traceD\|=1$ from~\eqref{eqn:B_norm_def} 
	and thus $\|\traceD'\|=1$. 
	Hence the operator norm of $\opN$ bounds that of $\opSL=\opN\traceD'$.
	To establish the $\opL$-harmonicity, 
	fix $g\in \Bd$ and set $v\coloneqq\opSL g=\opN\traceD' g$. 
	Then any $w\in \V_0=\ker\traceD$ satisfies
	\begin{align}
	\opl(v,w) 
	=\langle \opL v, \conj{w}\rangle_{\V'\times \V}
	= \langle \traceD' g, \conj{w}\rangle_{\V'\times \V}
	= \langle g, \conj{\traceD w}\rangle_{\Bd\times X}
	= 0 .%
	\end{align}
	Thus $v\in \V$ solves~\eqref{eqn:AMP} for $f\coloneqq\traceD v$. 
	The uniqueness follows from the well-posedness of~\eqref{eqn:AMP}:
	If $v_1,v_2\in \V$ solve~\eqref{eqn:AMP}, then $v_1-v_2\in \V_0$,~\eqref{eqn:L_coercive}, and~\eqref{eqn:L-harmonic}
	imply $v_1=v_2$.
\end{proof}

\Cref{thm:SL} suggests that the solutions to~\eqref{eqn:AMP}
are the images $\opSL\genvarBd\in \V$
of the single-layer potential for some unknown boundary density $\genvarBd\in \Bd$
that depends on the given data $\datavarB\in \B$
through the single-layer (boundary integral) operator 
\begin{align}\label{eqn:S_def}
	\opV\coloneqq \traceD\opSL:\Bd\to \B.
\end{align}
Hence $u\coloneqq \opSL\solvarBd$ solves~\eqref{eqn:AMP} for $\datavarB\in\B$
if and only if the boundary density $\solvarBd\in \Bd$ solves
\begin{align}\label{eqn:SP}
	\opV \solvarBd = \datavarB.
\end{align}
The single-layer problem~\eqref{eqn:SP}
can be interpreted as a generalised boundary integral equation of first kind.
The well-posedness follows from the coercivity~\eqref{eqn:L_coercive} of $\opL$.
\begin{theorem}[single-layer operator]\label{thm:S}
	The single-layer operator~\eqref{eqn:S_def} is bounded 
	with operator norm $\|\opV\|\coloneqq\|\opV\|_{L(\Bd;\B)}\leq \|\opN\|$ and
	coercive with
	\begin{align*}
		| \left\langle \opV\genvarBd, \conj{\genvarBd}\right\rangle_{\B\times \Bd}|
		\geq\frac{\const{\opL}}{\|\opL\|^2}\|\genvarBd\|_{\Bd}^2
		\qquad\text{for all }\genvarBd\in \Bd.
	\end{align*}
	In particular,~\eqref{eqn:SP} admits a unique solution for all $\datavarB\in \B$.
\end{theorem}
\begin{proof}
	The boundedness of $\opV$, with operator norm controlled by $\|\traceD\|=1$ and
	\Cref{thm:SL}, is clear.
	Consider any 
	$\genvarBd\in \Bd$ and set 
	$\genvarV\coloneqq \opSL \genvarBd=\opN\traceD'\genvarBd\in \V$
	such that $\opV\genvarBd=\traceD\genvarV$.
	Observe from $\opL\genvarV=\traceD'\genvarBd$ and
	the coercivity~\eqref{eqn:L_coercive} of $\opL$ that
	\begin{align*}
		| \left\langle \opV\genvarBd, \conj \genvarBd\right\rangle_{\B\times \Bd}|
		= |\left\langle \genvarV, \overline{\traceD'\genvarBd}\right\rangle_{\V\times \V'}|
		= |\opl(\genvarV,\genvarV)|
		\geq\const{\opL}\|\genvarV\|_{\V}^2
		\geq\frac{\const{\opL}}{\|\opL\|^2}\|\traceD'\genvarBd\|_{\V'}^2.
	\end{align*}
	Let $\genvarB\in \B$ be such that
	$\|\genvarBd\|_{\Bd} = 
	\left\langle \genvarBd, \conj{\genvarB}\right\rangle_{\Bd\times \B}$
	and $\|\genvarB\|_{\B}\leq 1$ ($X$ is a Hilbert space).
	Since $\traceD$ is an isometric isomorphism between 
	$\V_0^\perp\subset \V$ and $B$,
	there exists $R_{\traceD}\genvarB\in \V_0^\perp\subset \V$ with 
	\begin{align*}
		\traceD R_{\traceD}\genvarB = \genvarB
		\qquad\text{and}\qquad
		\|R_{\traceD}\genvarB\|_{\V}=\|\genvarB\|_{\B}\leq 1.
	\end{align*}
	(This is saying that the infimum in~\eqref{eqn:B_norm_def} is always attained for some $v\in \V$.)
	Hence the definition of the operator norm in $\V'$ 
	reveal
	\begin{align*}
		\|\traceD'\genvarBd\|_{\V'}
		=\sup_{0\ne\genvarV\in\V}
		\frac{|\left\langle \genvarBd,\conj{\traceD\genvarV}\right\rangle_{\Bd\times \B}|}
		{\|\genvarV\|_{\V}}
		\geq 
		\frac{|\left\langle \genvarBd,\conj{\traceD R_{\traceD}\genvarB}\right\rangle_{\Bd\times \B}|}
		{\|R_{\traceD}\genvarB\|_{\V}}
		\geq \left\langle \genvarBd,\conj{\genvarB}\right\rangle_{\Bd\times \B}
		=\|\genvarBd\|_{\Bd}.
	\end{align*}
	The combination of the displayed estimates conclude the proof.
\end{proof}
\begin{figure}[]
	\centering
	\begin{tikzpicture}[>=latex, font=\normalsize]
	  \node (X) at (0,0) {$\B^s$};
	  \node (W) at (4.2,0) {$\W^s$};
	  \node (B) at (2.1,1.7) {$\Bd^s$};
	  \draw[->] ([yshift=-0.07cm]X.east) -- node[below=2pt] {\small\eqref{eqn:AMP}} ([yshift=-0.07cm]W.west);
	  \draw[->] ([yshift=+0.07cm]W.west) -- node[above=2pt] {\small$\traceD$} ([yshift=+0.07cm]X.east);
	  \draw[<->] (B) -- node[above left] {\small $\opV$} (X);
	  \draw[<->] (B) -- node[above right] {\small $\opSL$} (W);
	\end{tikzpicture}
	\caption{Isomorphism between $\Bd^s,\B^s$, and $\W^s\subset \V$ for all 
		$0\leq s\leq \sigma_{reg}$.}
	\label{fig:abstrace_diagram}
\end{figure}
	\Cref{thm:SL,thm:S} imply that $\opSL:\Bd\to \W$ and $\opV:\Bd\to \B$ 
	are isomorphisms between the spaces $\Bd$, $\B$, and the set of 
	solutions to~\eqref{eqn:AMP} given by
	\begin{align*}
		\W
		\coloneqq\{v\in \V\ :\ \opL v = 0\text{ in } \V_0' \text{ in the sense of }\eqref{eqn:AMPa}\}
		\subset \V.
	\end{align*}
	Associate each of these spaces $Y\in\{\Bd,\B,\W\}$ with a scale $(Y^s)_{s\geq0}$
	of densely embedded Hilbert spaces indexed by some parameter $0<s$ such that
	\begin{align*}
		Y^s\hookrightarrow Y^t\hookrightarrow Y^0\coloneqq Y
		\qquad\text{for all }0<t<s.
	\end{align*}
	Then $\opSL$ and $\opV$ naturally extend to $\Bd^s$. Let $\sigma_{\mathrm{reg}}>0$ be such that
	\begin{align*}
		\opSL:\Bd^s\to \W^s
		\quad\text{and}\quad\opV:\Bd^s\to \B^s
		\qquad\text{for all }0\leq s\leq \sigma_{\mathrm{reg}}
	\end{align*}
	remain isomorphisms on the shifted scale up to $\sigma_{\mathrm{reg}}\geq0$.
	This leads to the commuting diagram depicted in \Cref{fig:abstrace_diagram} and 
	describes an abstract regularity shift for the solution to~\eqref{eqn:SP} 
	(without proof).%
	\begin{theorem}[regularity]\label{thm:regularity of solutions}
		For any $0\leq s\leq\sigma_{\mathrm{reg}}$ and
		$f\in \B^s$, the unique solution $\solvarBd\in \Bd$
		to~\eqref{eqn:SP} satisfies 
		$\solvarBd\in \Bd^s$ and $\opSL\solvarBd\in \W^s$.
		In particular, there exists $\Const{reg}(s)>0$ with
		\begin{align*}
			\Const{reg}(s)^{-1}\|\opSL\solvarBd\|_{\W^s}
			\leq\|\solvarBd\|_{\Bd^s}
			\leq \Const{reg}(s)\|\datavarB\|_{\B^s}.
		\end{align*}
	\end{theorem}

\begin{remark}[single-layer potential]\label{rem:SL}
	Operator representations of the classical layer potentials in terms of the Newton
	potential and trace operators already appear in%
	~\cite{Cos:BoundaryIntegralOperators1988} and are nowadays treated in
	textbooks~\cite{McL:StronglyEllipticSystems2000,SS:BoundaryElementMethods2011}.
	The definition of a single-layer potential by~\eqref{eqn:SL_def} 
	goes back at least to~\cite{Bar:LayerPotentialsGeneral2017} 
	and has been utilised to derive skeleton integral equations 
	for variable coefficient Helmholtz problems in%
	~\cite{FHS:SkeletonIntegralEquations2024,GHS:StableSkeletonIntegral2025}.
	Abstract definitions of the double layer potential 
	derived from the classical operator representation $\opDL=\opN\gamma_{\mathrm{N}}'$ 
	with the dual co-normal trace $\gamma_{\mathrm{N}}'$
	require a more involved functional analytical setup as discussed 
	in~\cite[Sec.~4.3]{GHS:StableSkeletonIntegral2025}, see also~\cite{Bar:LayerPotentialsGeneral2017}.
\end{remark}
\begin{remark}[jump relations]\label{rem:commuting_diagram}
	The application to second-order problems in 
	\Cref{sec:BEM for general single-layer operators} leads to $\V\subset H^1(\Omega)$
	and \Cref{fig:abstrace_diagram} ($s=0$) encapsulates
	the jump relations for the single-layer potential $\opSL$ in the following sense: 
	the Dirichlet jump vanishes by
	$\operatorname{im}\opSL=\W\subset H^1(\Omega)$ and the Neumann jump is its inverse 
	$\opSL^{-1}:\W\to B$, cf.~\cite[Thm.~4.3]{GHS:StableSkeletonIntegral2025}
	or~\Cref{lem:jump_relation} below for details.
\end{remark}

\subsection{Quasi-optimal boundary element discretisation}%
\label{sub:Abstract_quasi_optimal_BEM}
The boundary element method (BEM) for the discretisation of~\eqref{eqn:SP}
is based on a finite-dimensional conforming subspace $\Bdh\subset \Bd$.
The corresponding Galerkin discretisation leads to the natural BEM formulation of~\eqref{eqn:SP}
that seeks a discrete solution $\solvarBdh\in\Bdh$ to
\begin{align}\label{eqn:cDSP}
	\left\langle \opV \solvarBdh, \overline{\genvarBdh}\right\rangle_{\B\times \Bd}
	= \left\langle \datavarB,\overline{\genvarBdh}\right\rangle_{\B\times \Bd}
	\qquad\text{for all }\genvarBdh\in \Bdh.
\end{align}%
The Galerkin orthogonality implies the quasi-optimality of~\eqref{eqn:cDSP}:
its (discrete) solution is a quasi-best approximation of 
the unique solution $\solvarBd\in \Bd$ to~\eqref{eqn:SP}.
\begin{theorem}[quasi-best approximation of classical BEM]
	\label{thm:qb}
	The classical BEM formulation~\eqref{eqn:cDSP} is well-posed and its 
	unique solution $\solvarBdh\in \Bdh$ satisfies
	\begin{align*}
		\|\solvarBd-\solvarBdh\|_{\Bd}
		\leq \frac{\|\opN\|\|\opL\|^2}{\const{\opL}}
		\,\min_{\genvarBdh\in \Bdh}\|\solvarBd-\genvarBdh\|_{\Bd}.
	\end{align*}
\end{theorem}
\begin{proof}
	This is the C\'ea lemma 
	with $\|\opV\|\leq\|\opN\|$ from \Cref{thm:S}.
\end{proof}

The numerical realisation of~\eqref{eqn:cDSP} hinges on the possibility to evaluate
the single-layer operator $\opV$ on $\Bdh$.
However, this appears only feasible in special situations, 
where explicit analytical representations of $\opV$ are
available, cf.~\Cref{sub:Relation_to_classical_BEM} for an example.
A general BEM formulation therefore has to replace $\opV$ by some suitable approximation
$\opVh$ and it is desirable that the resulting formulation 
allows best-approximations as in \Cref{thm:qb}.

To this end, we consider a conforming subspace $\Vh\subset \V$ and the associated
Galerkin projector $P_\diamT:\V\to \Vh$ that maps any $\genvarV\in \V$ to the unique solution of
\begin{align}\label{eqn:Ph_def}
	\ell(P_\diamT \genvarV,\genvarVh) = \ell(\genvarV, \genvarVh)
	\qquad\text{for all }\genvarVh\in \Vh.
\end{align}
The discrete single-layer operator $\opVh:\Bd\to \B$ is then 
obtained by replacing the Newton potential
$\opN$ in~\eqref{eqn:S_def} with its Galerkin projection $P_\diamT\opN$, namely
\begin{align}\label{eqn:Sh_def}
	\opVh\coloneqq \traceD P_\diamT\opN\traceD'=\traceD P_\diamT\opSL.
\end{align}
The improvement over~\eqref{eqn:cDSP} is that~\eqref{eqn:Sh_def} allows an 
(algebraic) realisation in terms of the volume stiffness matrix 
as discussed in \Cref{sub:Algebraic formulation} below and
can therefore be interpreted as a computable counterpart of $\opV$ in~\eqref{eqn:cDSP}
associated to $\Vh\subset \V$.
The corresponding BEM formulation associated to $\opVh$ seeks a solution $\solvarBdh\in \Bdh$ to
\begin{align}\label{eqn:DSP}
	\left\langle \opVh \solvarBdh, \conj{\genvarBdh}\right\rangle_{\B\times \Bd}
= \left\langle \datavarB,\conj{\genvarBdh}\right\rangle_{\B\times \Bd}
	\qquad\text{for all }\genvarBdh\in \Bdh.
\end{align}

Before we state an analogous quasi-best approximation for~\eqref{eqn:DSP}
in \Cref{thm:qb} below, a discussion on the discrete consistency error
introduced by the approximation of the single-layer operator is of order, namely
\begin{align}\label{ass:simple_layer_approx}
	\errCon\coloneqq
	\|\opV-\opVh\|_{L(\Bdh;\Bdh')}
	\coloneqq
	\sup_{0\ne\genvarBdh,\genvarBdhalt\in\Bdh}
	\frac{\left|\left\langle (\opV-\opVh)\genvarBdh,\conj{\genvarBdhalt}\right\rangle_{\B\times
	\Bd}\right|}{\|\genvarBdh\|_{\Bd}\|\genvarBdhalt\|_{\Bd}}.
\end{align}
The point is that typical BEM spaces $\Bdh$ are more regular than $\Bd$ and
this is expressed~by
\begin{align*}
	\Bdh\subset \Bd^{s_*}\subset \Bd
	\qquad\text{for some }0\leq s_*.
\end{align*}
Hence~\Cref{fig:abstrace_diagram} suggests that $\Vh$ should approximate 
at least functions in $\W^{s_*}$ well.
For $0\leq s\leq s_*$,
we describe the approximation property of $\Vh$ in abstract notation by
\begin{align}\label{eqn:Vh_approx}
	\eta_s(\Vh)\coloneqq
	\sup_{\substack{w\in\W^s\\\|w\|_{\W^s}=1}}\inf_{\genvarVh\in\Vh}\|w-\genvarVh\|_{\V}
\end{align}
and encode a Bernstein-type property of $\Bdh$ by
\begin{align}\label{eqn:Bdh_inverse}
	\beta_s(\Bdh)\coloneqq
	\inf_{\substack{\genvarBdh\in\Bdh\\\|\genvarBdh\|_{\Bd^{s}}=1}}\|\genvarBdh\|_{\Bd}.
\end{align}
In the applications, $\eta_s(\Vh)$ and~$\beta_s(\Bdh)$
relate to $s$-powers of the maximal mesh-sizes 
of the underlying meshes associated to $\Vh$ and $\Bdh$ as discussed in \Cref{rem:consistency} below.
In this way,~\eqref{eqn:Vh_approx}--\eqref{eqn:Bdh_inverse} 
measure the resolution of the
volume and boundary discretisation, and their ratio controls 
the consistency error~\eqref{ass:simple_layer_approx} by the following key lemma.
\begin{lemma}[consistency error]\label{lem:consistency}
	Any $0\leq s\leq \max\{s^*,\sigma_{\mathrm{reg}}\}$ satisfies
	\begin{align}\label{eqn:SL_approx}
		\errCon
		\leq \Const{con}(s)\frac{\eta_s(\Vh)}{\beta_s(\Bdh)}
		\qquad\text{for }\Const{con}(s)
		\coloneqq \Const{reg}(s) \|\opL\|/\const{\opL}.
	\end{align}
\end{lemma}
\begin{proof}
	Let $\genvarBdh\in\Bdh\subset \Bd^{s}$ be arbitrary.
	C\'ea's lemma, the approximation property~\eqref{eqn:Vh_approx} of $\Vh$, 
	and the boundedness of $\opSL:\Bd^{s}\to \W^s$ with operator norm
	$\Const{reg}(s)$ provide
	\begin{align*}%
		\|(1-P_\diamT)\opSL\genvarBdh&\|_{\V}
		\leq\eta_s(\Vh) \frac{\|\opL\|}{\const{\opL}}\|\opSL\genvarBdh\|_{\W^s}
		\leq\Const{con}(s)\eta_s(\Vh) \|\genvarBdh\|_{\Bd^s}.
	\end{align*}
	A Bernstein-type inequality~\eqref{eqn:Bdh_inverse} reveals
	$\beta_s(\Bdh)\|(1-P_\diamT)\opSL\genvarBdh\|_{\V}\leq\Const{con}(s)\eta_s(\Vh)\|\genvarBdh\|_{\Bd}$.
	Hence the identity $\opV-\opVh = \traceD(1-P_\diamT)\opSL$ and $\|\traceD'\|=1$
	verify for any $\genvarBdhalt\in\Bdh$ that
	\begin{align*}
		\left|\left\langle (\opV-\opVh)\genvarBdh,
		\conj{\genvarBdhalt}\right\rangle_{\B\times\Bd} \right|
		&\leq \|(1-P_\diamT)\opSL \genvarBdh\|_{\V}
		\|\traceD'\genvarBdhalt\|_{\V'}
		\leq 
		\Const{con}(s)\frac{\eta_s(\Vh)}{\beta_s(\Bdh)}
		\|\genvarBdh\|_{\Bd}
		\|\genvarBdhalt\|_{\Bd}.\tag*{\qedhere}
	\end{align*}
\end{proof}

The abstract version of \textbf{\Cref{thm:A}} establishes the well-posedness and quasi-best approximation
for~\eqref{eqn:DSP} if the consistency error is small,
which is guaranteed for $\eta_s(\Vh)\ll\beta_s(\Bdh)$ by \Cref{lem:consistency}.
Recall the exact solution $\solvarBd\in\Bd$ from~\eqref{eqn:SP}.
\begin{theorem}[quasi-best approximation]\label{thm:quasi_best}
	If $\errCon\leq\varepsilon_{\opL}\coloneqq\const{\opL}/(2\|\opL\|^2)$, 
	then the discrete problem~\eqref{eqn:DSP}
	admits a unique solution $\solvarBdh\in\Bdh$ for any $\datavarB\in\B$ and
	\begin{align}\label{eqn:Sh_quasi_best}
		\Const{qb}^{-1}\|\solvarBd-\solvarBdh\|_{\Bd}
		\leq \min_{\genvarBdh\in \Bdh}\|\solvarBd-\genvarBdh\|_{\Bd}
			+  \min_{\genvarVh\in \Vh}\|\opSL\solvarBd-\genvarVh\|_{\V}.
	\end{align}
	The condition $\errCon\leq \varepsilon_{\mathrm{\opL}}$ holds, in particular,
	for any pair $\Vh\subset\V$ and $\Bdh\subset \Bd$ 
	that satisfies $\Const{con}(s)\eta_s(\Vh)\leq \varepsilon_{\opL}\beta_s(\Bdh)$ 
	for some $0\leq
	s\leq\max\{s_*,\sigma_{\mathrm{reg}}\}$.
	The constant $\Const{qb}>0$ is independent of the trace $\traceD$ and 
	exclusively depends on $\|\opL\|,\|\opN\|$, and
	$\const{\opL}$.
\end{theorem}
\begin{proof}
	Let $\errCon\leq \varepsilon_{\opL}$ 
	and observe that this follows from \Cref{lem:consistency}
	provided that $\Const{con}(s)\eta_s(\Vh)\leq \varepsilon_{\opL}\beta_s(\Bdh)$ holds 
	for some $0\leq s\leq\max\{s_*,\sigma_{\mathrm{reg}}\}$.
	Deduce from~\Cref{thm:S} that $\opVh$ is coercive on $\Bdh$ 
	with coercivity constant $\const{\opL}\|\opL\|^{-2}-\varepsilon_{\mathrm{con}}\geq 
	\varepsilon_{\opL}>0$.
	Hence~\eqref{eqn:DSP} admits a unique solution $\solvarBdh\in \Bdh$ that is considered in the following.

	Let $\Pi\solvarBd\in \Bdh$ denote the best approximation of the 
	exact solution $\solvarBd\in \Bd$ in the Hilbert space $\Bdh$ characterised by
	\begin{align}\label{eqn:quasi_best_orth_Bdh}
		\min_{\genvarBdh\in\Bdh}\|\solvarBd-\genvarBdh\|_{\Bd}^2
		=\|\solvarBd-\Pi\solvarBd\|_{\Bd}^2
		=\|\solvarBd-\solvarBdh\|_{\Bd}^2 - \|\solvarBdh-\Pi\solvarBd\|_{\Bd}^2.
	\end{align}
	Employ 
	the coercivity of $\opV$ from \Cref{thm:S} for 
	\begin{align}\label{eqn:quasi_best_split}
		\const{\opL}\|\opL\|^{-2}\|\solvarBd-\solvarBdh\|_{\Bd}^2
		&\leq \left|\left\langle \opV(\solvarBd-\solvarBdh), 
		\conj{\solvarBd-\Pi\solvarBd}\right\rangle_{\B\times \Bd}\right|
		+ \left|\left\langle \opV(\solvarBd-\solvarBdh), 
		\conj{\Pi\solvarBd-\solvarBdh}\right\rangle_{\B\times \Bd}\right|.
	\end{align}
	A Cauchy-Schwarz inequality bounds the first term in the right-hand side 
	of~\eqref{eqn:quasi_best_split} by
	\begin{align}\label{eqn:quasi_best_split_1}
		\left|\left\langle \opV(\solvarBd-\solvarBdh), 
		\conj{\solvarBd-\Pi\solvarBd}\right\rangle_{\B\times \Bd}\right|
		\leq\|\opV\|\|\solvarBd-\solvarBdh\|_{\Bd}
		\|\solvarBd-\Pi\solvarBd\|_{\Bd}.
	\end{align}
	To control the second term in the right-hand side of~\eqref{eqn:quasi_best_split}, 
	use the Galerkin orthogonality 
	$\left\langle \opV \solvarBd - \opVh \solvarBdh, \conj{e_\diamP}\right\rangle_{\B\times \Bd}=0$ for
	$e_\diamP\coloneqq \solvarBdh-\Pi\solvarBd\in\Bdh$ to obtain
	\begin{align}\nonumber
		\langle\opV
		&(\solvarBd-\solvarBdh),\conj{\Pi\solvarBd-\solvarBdh}\rangle_{\B\times \Bd}
		=\left\langle (\opV-\opVh)\solvarBdh, \conj{e_\diamP}\right\rangle_{\B\times \Bd}\\
		&= \left\langle (\opV-\opVh)\solvarBd, \conj{e_\diamP}\right\rangle_{\B\times \Bd}
		- \left\langle (\opV-\opVh)(\solvarBd-\Pi\solvarBd), \conj{e_\diamP}\right\rangle_{\B\times \Bd}
		+ \left\langle (\opV-\opVh)e_\diamP, \conj{e_\diamP}\right\rangle_{\B\times \Bd}.
		\label{eqn:quasi_best_split_2}
	\end{align}
	Recall the identity $\opV-\opVh = \traceD(1-P_\diamT)\opSL$,
	and observe from the Galerkin orthogonality of $P_\diamT$ that
	$\const{\opL}\|(1-P_\diamT)\opSL(\solvarBd-\Pi\solvarBd)\|_{\V}
	\leq \|\opL\|\|\opSL\|\|\solvarBd-\Pi\solvarBd\|_{\Bd}$.
	This and~\eqref{ass:simple_layer_approx} control the three terms 
	in the final equality of~\eqref{eqn:quasi_best_split_2} separately by
	\begin{align}\label{eqn:quasi_best_split_3a}
		\left|\left\langle (\opV-\opVh)\solvarBd, \conj{e_\diamP}\right\rangle_{\B\times \Bd}\right|
		&\leq \|(1-P_\diamT)\opSL\solvarBd\|_{\V}\|\traceD'e_\diamP\|_{\V'},\\
		\left|\left\langle (\opV-\opVh)\,(\solvarBd-\Pi\solvarBd), 
			\conj{e_\diamP}\right\rangle_{\B\times \Bd}\right|
		&\leq \frac{\|\opL\|}{\const{\opL}}\|\opSL\|\,
		\|\solvarBd-\Pi\solvarBd\|_{\Bd}\|\traceD'e_\diamP\|_{\V'},\\
		\left|\left\langle (\opV-\opVh)e_\diamP, 
			\conj{e_\diamP}\right\rangle_{\B\times \Bd}\right|
		&\leq \errCon\,\|e_\diamP\|_{\Bd}^2.\label{eqn:quasi_best_split_3c}
	\end{align}
	\Cref{thm:SL,thm:S} provide
	$\|\opV\|+\|\opSL\|\leq 2\|\opN\|$.
	Since $\|\traceD'e_\diamP\|_{\V'}\leq \|e_\diamP\|_{\Bd}\leq \|\solvarBd-\solvarBdh\|_{\Bd}$
	from $\|\traceD\|= 1$ and~\eqref{eqn:quasi_best_orth_Bdh},
	the combination of~\eqref{eqn:quasi_best_split}--\eqref{eqn:quasi_best_split_3c} 
	and $\const{\opL}\leq\|\opL\|$
	establish
	\begin{align*}
		(\const{\opL}\|\opL\|^{-2} - \errCon)\|\solvarBd-\solvarBdh\|_{\Bd}^2
		&\leq \left(2\frac{\|\opL\|}{\const{\opL}}\|\opN\|\|\solvarBd-\Pi\solvarBd\|_{\Bd}
		+ \|(1-P_\diamT)\opSL\solvarBd\|_{\V}\right)\|\solvarBd-\solvarBdh\|_{\Bd}.
	\end{align*}
	Hence C\'ea's lemma for $P_\diamT$ and the 
	Pythagoras identity~\eqref{eqn:quasi_best_orth_Bdh}
	reveal
	\begin{align*}
		(\const{\opL}\|\opL\|^{-2} - \errCon)\|\solvarBd-\solvarBdh\|_{\Bd}
		&\leq \frac{\|\opL\|}{\const{\opL}}
		\left(2\|\opN\|\min_{\genvarBdh\in \Bdh}\|\solvarBd-\genvarBdh\|_{\Bd}
		+ \min_{\genvarVh\in \Vh}\|\opSL\solvarBd-\genvarVh\|_{\V}\right).
	\end{align*}
	This and $\errCon\leq \const{\opL}/(2\|\opL\|^2)$ conclude the proof for 
	$\Const{qb}\coloneqq2\|\opL\|^3\max\{1,2\|\opN\|\}/\const{\opL}^2$.
\end{proof}

\begin{remark}[consistency]\label{rem:consistency}
	Typical finite and boundary element discretisations
	$\Vh$ and $\Bdh$ over (quasi-uniform) meshes of $\Omega$ and $\Sigma$ with maximal mesh-sizes $\diamT$
	and $\diamP$ satisfy~\eqref{eqn:Vh_approx}--\eqref{eqn:Bdh_inverse} with $\eta_s(\Vh)\approx\diamT^s$ 
	and $\beta_s(\Bdh)\approx\diamP^s$ for a certain range of $0\leq s\leq s^*$.
	A natural compatibility condition for $\Vh$ with the boundary discretisation is that
	$\diamT\in\mathcal{O}(\diamP)$ should be (at least) of the same order as $\diamP$.
	\Cref{thm:quasi_best} verifies that a linear relation indeed suffices to obtain well-posedness
	and quasi-optimality of~\eqref{eqn:DSP} if $\max\{s_*,\sigma_{\mathrm{reg}}\}>0$.
\end{remark}%
\subsection{Algebraic BEM formulation}%
\label{sub:Algebraic formulation}
To expose the matrix structure of~\eqref{eqn:DSP},
fix bases $\{\genvarVh[j]\}_{j=1}^{\NVh}\subset \Vh$ 
and $\{\genvarBdh[j]\}_{j=1}^{\NBdh}\subset \Bdh$ 
of size $\NVh=\dim \Vh$ and $\NBdh=\dim \Bdh$.
Define the $\Vh$-stiffness matrix $L_\diamT\in\mathbb C^{\NVh\times\NVh}$,
the connection matrix $R_{\diamT}^{\diamP}\in \mathbb C^{\NVh\times\NBdh}$, 
and the right-hand side vector $RHS\in \mathbb C^{\NBdh}$
by
\begin{align*}
	L_{\diamT, jk}
		\coloneqq \opl\left(\genvarVh[k], \genvarVh[j]\right),
			\quad
	R_{\diamT, jm}^{\diamP}
		\coloneqq \left\langle\traceD\genvarVh[j], \overline{\genvarBdh[m]}\right\rangle_{\B\times \Bd}
		\quad\text{and}\quad
	RHS_{m}
	\coloneqq \left\langle\datavarB, \overline{\genvarBdh[m]}\right\rangle_{\B\times \Bd}
\end{align*}
for all $j,k=1,\dots,\NVh$ and $m=1,\dots,\NBdh$.
These relate to their continuous counterparts by
the basis-induced isomorphisms $J_\diamT:\mathbb C^{\NVh}\to \Vh\subset \V$ 
and $J_\diamP:\mathbb C^{\NBdh}\to \Bdh\subset\Bd$ with
\begin{align*}
	J_\diamT x 
	&= \sum^{\NVh}_{j=1} x_j \genvarVh[j]\in \Vh
	&&\text{for all }x\in\mathbb C^{\NVh},\\
	J_\diamP z 
	&= \sum^{\NBdh}_{j=1} z_j \genvarBdh[j]\in \Bdh
	&&\text{for all }z\in\mathbb C^{\NBdh}.
\end{align*}
The identification of the dual spaces of $\mathbb C^{\NVh}$ and $\mathbb C^{\NBdh}$ 
with themselves through the Euclidean inner product
$\left\langle \bullet,\conj\bullet\right\rangle_2$
leads to the dual operators
$J_\diamT':\V'\to\mathbb C^{\NVh}$ and $J_\diamP':\B\to\mathbb C^{\NBdh}$ 
defined by~\eqref{eqn:adjoint_def}, e.g., 
$\left\langle J_\diamT'f,\conj x\right\rangle_2 
= \left\langle f, \conj{J_\diamT x}\right\rangle_{\V'\times \V}$
for all $f\in \V', x\in\mathbb C^{\NVh}$.
In particular,
\begin{align*}
	L_\diamT = J_\diamT'\opL\,J_\diamT
	,\qquad
	R_\diamT^{\diamP} = J_\diamT'\traceD'\,J_\diamP,
	\qquad\text{and}\qquad
	RHS = J_\diamP'\datavarB.
\end{align*}
Moreover, the Galerkin projection~\eqref{eqn:Ph_def}
satisfies $P_\diamT= J_\diamT L_\diamT^{-1} J_\diamT'\opL$ and hence
$P_\diamT\opN = J_\diamT L_\diamT^{-1} J_\diamT'$.
With $N_\diamT\coloneqq L_\diamT^{-1}$,
this reveals the representation of 
$\opVh$ from~\eqref{eqn:Sh_def} as
\begin{align}\label{eqn:Sh_def_alt}
	\opVh = \traceD J_\diamT N_\diamT J_\diamT'\traceD':\Bd\to \B.
\end{align}
Using~\eqref{eqn:Sh_def_alt} and the connection matrix $R^\diamP_\diamT$, 
the discrete solution $\solvarBdh=J_\diamP z\in \Bdh$ to~\eqref{eqn:DSP}
is characterised by its coefficient vector $z\in\mathbb C^{\NBdh}$ solving
\begin{align}\label{eqn:Vhef}
	V_{\diamT}^{\diamP} z = RHS
	\qquad\text{with}\qquad
V_{\diamT}^{\diamP}\coloneqq (R_{\diamT}^{\diamP})^{H}N_{\diamT}R_{\diamT}^\diamP.
\end{align}
Note that the definition in (\ref{eqn:Vhef}) provides a fully discrete description of $V_{d}^{h}%
\in\mathbb{C}^{n_{h}\times n_{h}}$. 
\begin{remark}[localisation and dimension reduction]\label{rem:abstract_localisation}
	In typical applications, 
	$\Vh$ and $\Bdh$ admit locally supported (e.g., finite and boundary element) bases.
	Then the connection matrix $R_{\diamT}^{\diamP}$ is sparse: its columns only couple to
	volume basis functions with nonzero trace in $X$. 
	Consequently, in the factorisation~\eqref{eqn:Vhef} of $V_\diamT^\diamP$,
	the inverse stiffness matrix $N_{\diamT}$ is accessed only through its action on
	the range of $R_{\diamT}^{\diamP}$, i.e., the degrees of freedom with nonzero trace. 
	Hence the discrete problem
	\eqref{eqn:Vhef} retains the characteristic dimension reduction of boundary integral formulations: 
	instead of solving on a volume discretisation of size $\NVh$,
	only a nonlocal discrete equation of size $\NBdh$ on the interface has to be solved.
\end{remark}

Besides the dimension reduction described in \Cref{rem:abstract_localisation}, the
factorisation of $V_\diamT^\diamP$ in \eqref{eqn:Vhef} also supports efficient 
\emph{approximate} realisations by replacing 
$N_\diamT=L_\diamT^{-1}$ with a computable surrogate, avoiding an exact inversion.
Assume that $\Nh\in \mathbb C^{\NVh\times\NVh}$ satisfies
\begin{align}\label{eqn:Nh_approx_error}
	\|J_\diamT\|^2\|\Nh - N_\diamT\|_{2}\leq \varepsilon_\diamT,
\end{align}
where 
the operator norm $\|J_\diamT\|$ of 
$J_\diamT:\mathbb{C}^{\NVh}\to \V$ accounts for the scaling induced by the chosen basis of $\Vh$.
This leads to a numerical approximation of $\opVh$ (and therefore of $\opV$)~by
\begin{align}\label{eqn:Sht_def}
	\opVht \coloneqq \traceD J_\diamT \Nh J_\diamT'\traceD':\Bd\to \B
\end{align}
with an analogous quasi-best approximation property of the corresponding single-layer discretisation.
Recall $\errCon$ from~\eqref{ass:simple_layer_approx},
and $\varepsilon_{\opL}$ and $\Const{qb}$ from \Cref{thm:quasi_best}.%
\begin{cor}[quasi-best approximation]\label{cor:quasi-best approximation}
	Consider any $\datavarB\in \B$ with solution $\solvarBd\in \Bd$ to~\eqref{eqn:SP}.
	If $\varepsilon_\diamT+\errCon\leq \varepsilon_{\opL}$,
	then the discrete problem~\eqref{eqn:DSP} 
	with $\opVh$ replaced by $\opVht$ 
	has a unique solution $\solvarBdh\in \Bdh$ with
	\begin{align*}
		\Const{qb}^{-1}\|\solvarBd-\solvarBdh\|_{\Bd}
		\leq \min_{\genvarBdh\in \Bdh}\|\solvarBd-\genvarBdh\|_{\Bd}
			+ \min_{\genvarVh\in \Vh}\|\opSL\solvarBd-\genvarVh\|_{\V}
			+\varepsilon_\diamT\,\|\solvarBd\|_{\Bd}.
	\end{align*}
\end{cor}
\begin{proof}
	The representations~\eqref{eqn:Sh_def_alt} and~\eqref{eqn:Sht_def} of
	$\opVh$ and $\opVht$ reveal
	for any $\genvarBd,\genvarBdalt\in \Bd$ that
	\begin{align*}
		\left|\left\langle(\opVh-\opVht)\genvarBd,
			\conj{\genvarBdalt}\right\rangle_{\B\times \Bd}\right|
		&=\left|\left\langle (N_\diamT-\Nh)J_\diamT'\traceD'\genvarBd,
		\conj{J_\diamT'\traceD'\genvarBdalt}\right\rangle_{2}\right|\\
		&\leq \|N_\diamT-\Nh\|_{2}\|J_\diamT\|^2\|\genvarBd\|_{\Bd}\|\genvarBdalt\|_{\Bd}
	\end{align*}
	with $1=\|\traceD'\|$ and $\|J_\diamT\| = \|J_\diamT'\|$ in the last step.
	This and~\eqref{eqn:Nh_approx_error} provide
	\begin{align}\label{eqn:opVh_approx}
		\|\opVh-\opVht\|_{L(\Bd;\B)}
		\leq \varepsilon_\diamT.
	\end{align}
	Repeat the first steps in the proof of \Cref{thm:quasi_best} with the same notation
	verbatim until~\eqref{eqn:quasi_best_split_2}, where
	the Galerkin orthogonality now holds for $\opVht$ instead of $\opVh$ and leads
	to the additional term 
	$\big\langle (\opVht-\opVh)\solvarBdh,\conj{e_\diamP}\big\rangle_{\B\times \Bd}$
	with $e_\diamP=\solvarBdh-\Pi\solvarBd$.
	This term is controlled by~\eqref{eqn:opVh_approx},
	a triangle inequality, and $\|e_\diamP\|_{\Bd}\leq\|\solvarBd-\solvarBdh\|_{\Bd}$
	as before by
	\begin{align*}
		\left| \left\langle (\opVht-\opVh)\solvarBdh, 
		\conj{e_\diamP}\right\rangle_{\B\times \Bd} \right|
		\leq \varepsilon_\diamT\, \left(\|\solvarBd\|_{\Bd}+\|\solvarBd-\solvarBdh\|_{\Bd}\right)
		\|\solvarBd-\solvarBdh\|_{\Bd}.
	\end{align*}
	The remaining parts of the proof of \Cref{thm:quasi_best} 
	apply with straightforward modifications; further
	details are omitted.
\end{proof}

\section{BEM for general second-order problems}%
\label{sec:BEM for general single-layer operators}
This section introduces a boundary element method (BEM) for the discretisation of
the single-layer operator associated to general linear second-order differential operators.

\subsection{Model second-order problem}%
\label{sub:Model problem}
Consider a Lipschitz domain $\Omega\subset\R^n$ for $n\geq 2$ 
with relatively closed %
$\GammaD\subset\partial\Omega$ and set $\V\coloneqq H^{1}_{\GammaD}(\Omega)$.
Given $F\in \V'$, the second-order model problem seeks a solution $u\in \V$~to
\begin{subequations}\label{eqn:2nd_PDE}
	\begin{align}\label{eqn:2nd_PDEa}
		-\Div(\opA\nabla u)  + \mathbf{b}\cdot \nabla u + c u
		&=F
		&&\text{in }\Omega,\\\label{eqn:2nd_PDEb}
		\traceN[\Omega] u
		&= Gu
		&&\text{on }\Gamma_{\mathrm{N}}\coloneqq\partial\Omega\setminus\GammaD
	\end{align}
\end{subequations}
with coefficients $\opA\in L^\infty(\Omega; \mathbb{H}^n), \mathbf{b}\in L^\infty(\Omega;\C^n), c\in
L^\infty(\Omega)$, and a bounded linear operator 
$G\in L(\V;H^{-1/2}(\Gamma_{\mathrm{N}}))$.
Standard examples of~\eqref{eqn:2nd_PDEb}
are $G\equiv0$ for Neumann, $G(u) = r u$ with $r\in \mathbb{C}$ for Robin,
and $G(u) = \operatorname{DtN}(u)$ for transparent (also called
Dirichlet-to-Neumann)%
~\cite{Fen:FiniteElementMethod1983,KG:ExactNonreflectingBoundary1989,
GS:DirichlettoNeumannOperatorHelmholtz2025}
boundary conditions on $\Gamma_{\mathrm{N}}$.
The leading coefficient $\opA$ is Hermitian and strongly elliptic: 
There exists $a_{\min}>0$ with
\begin{align*}
	a_{\min}|\xi|^2\leq \Re \xi^{H}\opA(x)\xi
	\qquad\text{for all }\xi\in\mathbb{C}^n
	\text{ and almost every }x\in\Omega.
\end{align*}
Further conditions on the coefficients are encapsulated in the coercivity 
requirement~\eqref{eqn:L_coercive} for the 
sesquilinear form $\opl:\V\times \V\to \mathbb{C}$
and for $\opL:\V\to \V'$
associated to~\eqref{eqn:2nd_PDE}, namely
\begin{align*}%
	\opl(\genvarV,\genvarValt)
	&\coloneqq \left\langle \opL \genvarV, \conj{\genvarValt}\right\rangle_{\V'\times \V}
	\coloneqq\int_{\Omega}\opA\nabla \genvarV\cdot \nabla \conj{\genvarValt}
	+ \left(\mathbf{b}\cdot \nabla \genvarV + c \genvarV\right)\,\conj{\genvarValt} \d x
	- \left\langle Gv, \traceD[\partial\Omega]\conj w\right\rangle_{\partial\Omega}
\end{align*}
for all $\genvarV,\genvarValt\in \V$.
The bounded inverse $\opN\coloneqq\opL^{-1}:\V'\to \V$
generalises the classical Newton potential%
~\cite[Sec.~3]{SS:BoundaryElementMethods2011}
even without assuming the existence of a Green's function for~\eqref{eqn:2nd_PDE}.

A Lipschitz interface $\Sigma\subset\Omega\cup\GammaN$ 
induces the trace space $\B\coloneqq H^{1/2}(\Sigma)$
as the image of the classical trace map 
$\traceS:H^1_{\Gamma_{\mathrm{D}}}(\Omega)\to H^{1/2}(\Sigma)$.
We assume for simplicity that either
\begin{enumerate}[label=(\alph*)]
	\item $\Sigma\subset\Omega$ decomposes $\Omega=\Omega_+\cup \Sigma\cup\Omega_-$ 
	into disjoint Lipschitz sets $\Omega_\pm\subset\Omega$, or
	\item $\Sigma\subset\GammaN(\subset\partial\Omega)$ is 
	an isolated boundary component of $\Omega_+\coloneqq \Omega$ 
	and we set $\Omega_-\coloneqq\emptyset$.
\end{enumerate}
The abstract transmission problem~\eqref{eqn:AMP} takes the form:
Given $\datavarB\in \B=H^{1/2}(\Sigma)$, find a (weak) solution $u\in \V$ to 
\begin{subequations}\label{eqn:MP}
	\begin{align}\label{eqn:MPa}
		-\Div(\opA\nabla u)  + \mathbf{b}\cdot \nabla u + c u
		&=0
		&&\text{in }\Omega\setminus\Sigma,\\\label{eqn:MPb}
		\traceN[\Omega]u
		&= Gu
		&&\text{on }\GammaN\setminus\Sigma,\\
		\traceS u 
		&=\datavarB
		&&\text{on }\Sigma.\label{eqn:MPc}
	\end{align}
\end{subequations}
The well-posedness is clear from the discussion in \Cref{sub:Abstract_model_problem}
and guarantees a unique solution $u\in \V$ to~\eqref{eqn:MP}.
Indeed, the single-layer potential and operator
\begin{align*}
	\opSL&=\opN\traceS':\Bd\to \V,\\
	\opV&=\traceS\opSL:\Bd\to \B
\end{align*}
defined on the dual trace space $\Bd=\B'=H^{-1/2}(\Sigma)$
characterise $u=\opSL\solvarBd$ in terms of the unique solution
$\solvarBd\in\Bd$ to the single-layer equation~\eqref{eqn:SP}.
Recall from \Cref{rem:commuting_diagram} that the single-layer potential is bijective onto the space
$\W$ characterised by
\begin{align}\label{eqn:W_def}
	\W=\left\{v\in \V\ :\ 
		\begin{aligned}
			\Div(\opA\nabla v)  &= \mathbf{b}\cdot \nabla v + c v 
								&&\text{in }\Omega\setminus\Sigma,\\
			\traceN[\Omega]v&= Gv &&\text{on }\GammaN\setminus\Sigma
		\end{aligned}
	\right\}\subset H^1(\Omega\setminus\Sigma,\opA).
\end{align}
Its inverse extends to the Neumann jump 
$\jumpN\bullet: H^1(\Omega\setminus\Sigma,\opA)\to H^{-1/2}(\Sigma)$ defined by
\begin{align*}
	\jumpN v\coloneqq 
	\left(\traceN[\Omega_+]v\right)|_{\Sigma} + \left(\traceN[\Omega_-]v\right)|_{\Sigma}
	\in H^{-1/2}(\Sigma)
\end{align*}
with the convention $\traceN[\Omega_-]v\coloneqq -Gv$  if $\Omega_-=\emptyset$.

\begin{lemma}[jump relation]\label{lem:jump_relation}
	Any $\genvarBd\in\Bd$ satisfies the \emph{jump relation}
	$\jumpN{\opSL\genvarBd} = \genvarBd$ and any $v\in \W$ satisfies 
	the single-layer \emph{Green's representation formula} 
	$v = \opSL\jumpN v$.
\end{lemma}

\begin{proof}
	Recall $\operatorname{dist}(\partial\Omega\setminus\Sigma,\Sigma)>0$ by assumption on $\Sigma$.
	Green's identity on $\Omega_{\pm}$ and the characterisation~\eqref{eqn:W_def} of $\W$
	reveal for any $\genvarBd\in\Bd$ with $v\coloneqq\opSL\genvarBd\in \W$ and $w\in\V$ that
	\begin{align*}
		\left\langle \jumpN{v}, \traceS\conj w\right\rangle_{\Sigma}
		&=\int_{\Omega\setminus\Sigma}\opA\nabla v\cdot \nabla\conj w 
		+ \left(\mathbf{b}\cdot \nabla \genvarV + c \genvarV\right)\,\conj{\genvarValt} \d x
		-\left\langle G{v}, \traceD[\partial\Omega]\conj w\right\rangle_{\partial\Omega} = \opl(v,w).
	\end{align*}
	Hence the claim follows from the surjectivity and definition of $\opSL=\opN\traceS'$ that implies
	\begin{align*}
		\left\langle \jumpN{{\opSL \genvarBd}}, \traceS\conj w\right\rangle_{\Sigma}
		= \opl({\opSL \genvarBd},w) = \left\langle \opL {\opSL \genvarBd},\conj w\right\rangle _{\V'\times V}
		= \left\langle g,\traceD[\Sigma]\conj w\right\rangle_{\Sigma}.
	\tag*{\qedhere}
	\end{align*}
\end{proof}

This model second-order problem setting fits into the abstract framework of
\Cref{sec:BEM_for_the_single-layer_potential}.\linebreak
\textbf{\Cref{thm:A}} follows for any discretisation pair
$\Vh\subset\V$ and $\Bdh\subset\Bd$ from \Cref{thm:quasi_best}
under the explicit compatibility condition 
$\varepsilon_{\mathrm{con}}\leq\varepsilon_{\opL}\coloneqq\const{\opL}/(2\|\opL\|^2)$.
In order to interpret the compatibility condition as a condition on the relative resolution of the 
volume and boundary discretisations by \Cref{lem:consistency}, a discussion on the elliptic regularity is of order.

\subsection{Mapping properties and elliptic regularity}%
\label{sub:Elliptic regularity}

The natural scale $0\leq s$ of Sobolev spaces associated to the operators 
$\opSL$ and $\opV$ reads
\begin{align*}
	\Bd^s\coloneqq H^{-1/2+s}(\Sigma),
	\qquad
	\B^s\coloneqq H^{1/2+s}(\Sigma),
	\qquad\text{and}\qquad
	\W^s\coloneqq H^{1+s}(\Omega\setminus\Sigma)\cap \W.
\end{align*}
Here, the possible range for $s$ is restricted by the smoothness of the interface 
(e.g., $s\leq 1$ for merely Lipschitz $\Sigma$), and 
$\sigma_{\mathrm{reg}}\geq0$ indicates the range $0\leq s\leq\sigma_{\mathrm{reg}}$
for which %
\begin{align*}
	\opSL:\Bd^s\to \W^s
	\qquad\text{and}\qquad
	\opV:\Bd^s\to \B^s
\end{align*}
are bijections as depicted in \Cref{fig:abstrace_diagram} 
and allow a shift theorem (\Cref{thm:regularity of solutions}).
In other words, $\sigma_{\mathrm{reg}}\geq0$ relates to an index of elliptic regularity
for~\eqref{eqn:MP} and thus depends on the regularity of the coefficients 
and also on the smoothness of $\partial\Omega$ and $\Sigma$.
The following theorem characterises $\sigma_{\mathrm{reg}}\geq0$ in a  
model setting with $G\equiv 0$ and 
\Cref{rem:mapping_generalise,rem:higher_regularity} below discuss extensions.
Recall the assumptions on the PDE~\eqref{eqn:2nd_PDE} from \Cref{sub:Model problem}.

\begin{theorem}\label{thm:mapping_property}
	Let $\partial\Omega$ and $\Sigma$ be of class $C^{k,1}$ with $k\in\mathbb{N}_0$
	and $G\equiv 0$ in~\eqref{eqn:2nd_PDE}.
	Suppose $\opA\in C^{k,1}(\Omega;\mathbb{H}^n)$ and
	$\mathbf{b}\in C^{k,1}(\Omega;\mathbb{C}^n)$ and, provided $k\geq1$,
	$c\in C^{k-1,1}(\Omega)$. 
	Then 
	\begin{enumerate}[label=(\roman*)]
		\item if $k=0$,
		then $\sigma_{\mathrm{reg}}\coloneqq \kappa$ is admissible for any $\kappa<1/2$ and
		\item if $k\geq1$,
			then $\sigma_{\mathrm{reg}}\coloneqq k$ is admissible.
	\end{enumerate}
\end{theorem}
\begin{proof}
	This proof derives the result for $s=\sigma_{\mathrm{reg}}$ 
	from classical elliptic regularity theory%
	~\cite{Gri:EllipticProblemsNonsmooth1985,Sav:RegularityResultsElliptic1998,McL:StronglyEllipticSystems2000,Agr:SpectralProblemsSecondorder2002}
	and the full range $0\leq s\leq \sigma_{\mathrm{reg}}$ follows by interpolation.

	\medskip
	\noindent\emph{Step 1} verifies the boundedness of 
	\begin{align}\label{eqn:bdd_injection}
		\opSL:H^{-1/2+\sigma_{\mathrm{reg}}}(\Sigma)\to H^{1+\sigma_{\mathrm{reg}}}(\Omega\setminus\Sigma)\cap \W,
		\qquad
		\opV:H^{-1/2+\sigma_{\mathrm{reg}}}(\Sigma)\to H^{1/2+\sigma_{\mathrm{reg}}}(\Sigma).
	\end{align}

	In the case $k=0$ with $\sigma_{\mathrm{reg}}=\kappa<1/2$,
	the trace 
	$\traceS:H^{1\pm\kappa}_{\Gamma_{\mathrm{D}}}(\Omega)\to H^{1/2\pm\kappa}(\Sigma)$
	is bounded on Lipschitz sets and so is the dual trace
	$\traceS':H^{-1/2+\kappa}(\Sigma)\to H^{-1+\kappa}(\Omega)\cap \V'$.
	Moreover $H^{1+\kappa}(\Omega\setminus\Sigma)\cap H^{1}(\Omega)=H^{1+\kappa}(\Omega)$.
	Since elliptic regularity asserts that
	\begin{align}\label{eqn:add_reg_L}
		\opL:H^{1+\kappa}(\Omega)\cap \V\to H^{-1+\kappa}(\Omega)\cap \V'
	\end{align}
	is an isomorphism, the identities $\opSL=\opN\traceS'$ and $\opV=\traceS \opSL$
	verify~\eqref{eqn:bdd_injection} for $\sigma_{\mathrm{reg}}=\kappa$.

	For $\sigma_{\mathrm{reg}}=k\geq1$ and any $\genvarBd\in H^{-1/2+k}(\Sigma)$, 
	$v\coloneqq \opSL\genvarBd\in \W$
	is $\opL$-harmonic in $\Omega\setminus\Sigma$ and satisfies
	$\jumpN v=\genvarBd$ by \Cref{lem:jump_relation}.
	Hence~\cite[Thm.~4.18 and 4.20]{McL:StronglyEllipticSystems2000}
	deduces $v=\opSL\genvarBd\in H^{1+k}(\Omega\setminus\Sigma)$.
	This and the regularity of the trace on $C^{k,1}$ surfaces%
	~\cite[Thm.~3.37]{McL:StronglyEllipticSystems2000} prove~\eqref{eqn:bdd_injection}
	for $\sigma_{\mathrm{reg}}=k$.

	\medskip
	\noindent\emph{Step 2} discusses the surjectivity on the shifted scale.
	\Cref{lem:jump_relation} identifies
	the inverse of $\opSL$ as the normal jump
	that extends to a bounded linear map 
	$\jumpN\bullet:H^{1+\sigma_{\mathrm{reg}}}(\Omega\setminus\Sigma)\cap \W\to
	H^{-1/2+\sigma_{\mathrm{reg}}}(\Sigma)$ 
	by the additional regularity of the coefficients and interface.
	This is straightforward for $\sigma_{\mathrm{reg}}=k\geq1$ and follows from the characterisation
	\begin{align*}
		\left\langle \jumpN v,\traceS\conj w\right\rangle_{\Sigma}
		=\opl(v,w) = \left\langle \opL v, \conj w\right\rangle_{\V'\times \V}
		\qquad\text{for all }v\in \W, w\in \V
	\end{align*}
	and the additional regularity~\eqref{eqn:add_reg_L} of $\opL$ if 
	$\sigma_{\mathrm{reg}}=\kappa<1/2$.
	Hence $\opSL$ from~\eqref{eqn:bdd_injection} is surjective.

	To establish the surjectivity of $\opV$, we argue that~\eqref{eqn:AMP}
	satisfies a regularity shift as follows.
	Consider any $\datavarB\in H^{1/2+\sigma_{\mathrm{reg}}}(\Gamma)$ and
	the unique solution $u\in \W$ to~\eqref{eqn:AMP}, which is~\eqref{eqn:MP} 
	for $G\equiv 0$ in the current setting.
	By assumption, $\Omega=\Omega_+\cup \Sigma\cup \Omega_-$ splits
	into disjoint Lipschitz sets $\Omega_\pm\subset \Omega$ 
	(possibly $\Omega_-=\emptyset$).
	The restrictions
	$u_\pm\coloneqq u|_{\Omega_\pm}
	\in H^{1}_{\GammaD\cap \partial\Omega_\pm}(\Omega_\pm)$
	of $u\in \W$ are uniquely given~by
	\begin{subequations}\label{eqn:MP_pm}
		\begin{align}\label{eqn:MP_pma}
			-\Div(\opA\nabla u_\pm)  + \mathbf{b}\cdot \nabla u_\pm + c u_\pm
			&=0
			&&\text{in }\Omega_\pm\setminus\Sigma,\\\label{eqn:MP_pmb}
			\traceN[\Omega_\pm]u_\pm
			&= 0
			&&\text{on }\GammaN\cap\partial\Omega_\pm\setminus\Sigma,\\
			\traceS u_\pm 
			&=\datavarB
			&&\text{on }\Sigma.\label{eqn:MP_pmc}
		\end{align}
	\end{subequations}
	Elliptic regularity 
	asserts $u_\pm\in H^{1+\sigma_{\mathrm{reg}}}(\Omega_\pm)$ and implies
	$u\in H^{1+\sigma_{\mathrm{reg}}}(\Omega\setminus\Sigma)\cap \W$.
	This and \Cref{fig:abstrace_diagram} reveal the surjectivity of $\opV$ 
	from~\eqref{eqn:bdd_injection} and conclude the proof.
\end{proof}

\begin{remark}[generalisations of \Cref{thm:mapping_property}]\label{rem:mapping_generalise}
	The smoothness assumptions on the coefficients in \Cref{thm:mapping_property}.i may be relaxed,
	e.g., to be multipliers on $H^t$ for some $0<t<1/2$.
	Then $\opL$ from~\eqref{eqn:add_reg_L}
	is bounded for all $0\leq \kappa\leq t$
	and an isomorphism at least for all $0\leq \kappa\leq \varepsilon $ 
	with some small $\varepsilon<t$%
	~\cite{Joc:$H^s$regularityResultGradient1999,HMW:HigherRegularitySolutions2019}.
	(The value of $\varepsilon$ can be quantified in terms of the operator norms of $\opL$ at $s\in\{0,t\}$
	and $\opN=\opL^{-1}$ at $s=0$.)
	Using this and the analogous regularity shift for~\eqref{eqn:MP_pm}, one obtains
	$\sigma_{\mathrm{reg}}=\varepsilon$ for some small $0<\varepsilon$.
	
	The extension of \Cref{thm:mapping_property} to more general boundary conditions ($G\ne0$) 
	is straightforward if the corresponding elliptic regularity shifts used in the proof holds.
	Examples include Robin-type boundary conditions $G(u)= r u$ with sufficiently smooth
	$r: \Sigma\to \mathbb{C}$, where%
	~\eqref{eqn:2nd_PDEb} and~\eqref{eqn:MPb}
	may be interpreted as inhomogeneous Neumann boundary conditions 
	with smooth Neumann data $G(u)\in H^{1/2+s}(\Sigma)$ (if $u\in H^{1+s}$)
	and standard elliptic regularity shifts%
	~\cite[Thm.~2.9]{Agr:SpectralProblemsSecondorder2002}.
	Another class of important boundary conditions
	are transparent (also called Dirichlet-to-Neumann) conditions%
	~\cite{Fen:FiniteElementMethod1983,KG:ExactNonreflectingBoundary1989,GS:DirichlettoNeumannOperatorHelmholtz2025,GHS:StableSkeletonIntegral2025}
	that enforce a compatibility condition of solutions with some (generalised) harmonic extension
	across $\GammaN$. 
	In other words, solutions are restrictions of those to a similar problem posed on a larger domain 
	from which the regularity properties are inherited.
\end{remark}
\begin{remark}[higher-regularity for Lipschitz domains]\label{rem:higher_regularity}
	If the coefficients are sufficiently smooth, $0<\sigma_{\mathrm{reg}}$
	depends exclusively on the regularity of $\partial\Omega$ and $\Sigma$.
	For piecewise smooth boundaries as in the case of polyhedral domains 
	$\Omega,\Omega_\pm\subset\R^n$, elliptic regularity 
	depends on the minimum of the interior angles%
	~\cite{Gri:EllipticProblemsNonsmooth1985,Gri:SingularitiesBoundaryValue1992}
	of $\Omega,\Omega_\pm$ and leads to some $1/2<\sigma_{\mathrm{reg}}$,
	where $H^{1/2+s}(\Sigma)$ for $s>1/2$ is understood as the trace of 
	$H^{1+s}(\Omega)$.
	If $\Omega=\Omega_+$ is convex (so that $\Omega_-=\emptyset$), one can choose
	$\sigma_{\mathrm{reg}}=1$. 
\end{remark}

\subsection{Relation to Green's function}%
\label{sub:Relation_to_classical_BEM}
There exists a vast literature on
the existence of a Green's function for~\eqref{eqn:2nd_PDE}
at least for pure Dirichlet boundary conditions ($\GammaD=\partial\Omega$)
and further assumptions on $\Omega$ and the coefficients $\mathbf{b}, c$, see, e.g.,%
~\cite{LSW:RegularPointsElliptic1963,GW:GreenFunctionUniformly1982,DM:EstimatesGreensMatrices1995,Beb:EfficientInversionGalerkin2004,HK:GreenFunctionEstimates2007,SS:BoundaryElementMethods2011,KS:GreensFunctionSecond2019,MP:ExponentialDecayEstimates2019}.

A Green's function for~\eqref{eqn:2nd_PDE} is an integral kernel
${\mathcal G}:\Omega\times\Omega\to \C\cup\{+\infty\}$ for the solution operator $\opN$ 
that incorporates the boundary conditions and satisfies%
\begin{align*}%
	\opN f
	= \int_{\Omega}{\mathcal G}(\bullet,y) f(y)\d y
	\qquad\text{and}\qquad
	\opN' f = \int_{\Omega}{\conj{{\mathcal G}(x,\bullet)}}f(x)\d x
	\qquad\text{for all }f\in C^\infty_0(\Omega),
\end{align*}
where the formula for the $L^2(\Omega)$-dual operator $\opN'$ follows from
Fubini's theorem. %
From the definition of $\opSL=\opN\traceS'$ in~\eqref{eqn:SL_def}, one infers
for any $\genvarBd\in L^1(\Sigma)$ that
\begin{align*}
	\left\langle \opSL \genvarBd, \conj{f}\right\rangle_{\V\times \V'}
	=\left\langle \genvarBd, \traceD \conj{\opN' f}\right\rangle_{\Bd\times B}
	=\int_{\Sigma}\int_{\Omega} {\mathcal G}(x,y)\genvarBd(y)\conj{f(x)}\d x \d y
	\qquad\text{for all }f\in C^\infty_0(\Omega).
\end{align*}
By identification of distributions, this provides the classical integral representation
of the single-layer potential (cf.~\cite[Sec.~3.1]{SS:BoundaryElementMethods2011})
\begin{align*}%
	\opSL\genvarBd = \int_{\Sigma}{\mathcal G}(\bullet,y) \genvarBd(y)\d y
	\qquad\text{for all }\genvarBd\in L^1(\Sigma)
\end{align*}
and, by restriction to $\Sigma$, the single-layer boundary integral operator $\opV=\traceS\opSL$.

If the Green's function ${\mathcal G}$ associated to $\opL$ is available explicitly,
then the identity
\begin{align}\label{eqn:cDSP_kernel}
	\left\langle \opV\genvarBdh,\conj{\genvarBdhalt}\right\rangle_{\B\times \Bd}
	=\int_{\Sigma}\int_{\Sigma}{\mathcal G}(x,y)\genvarBdh(y)\conj{\genvarBdhalt(x)}\d y\d x
	\qquad\text{for all }\genvarBdh,\genvarBdhalt\in \Bdh
\end{align}
makes the classical Galerkin BEM~\eqref{eqn:cDSP} feasible,
provided the double integral can be evaluated.
In practice, this requires a suitable numerical quadrature
adapt to the kernel singularity.
However, explicit representations for ${\mathcal G}$ are restricted to special cases and simple geometries.
The most prominent examples are \emph{constant} coefficients in the full space $\Omega=\R^n$
(where ${\mathcal G}$ is the fundamental solution)~\cite[Subsec.~3.1]{SS:BoundaryElementMethods2011}
and in the half space
$\Omega=\R^n_+$~\cite{CH:UniformlyValidFar1995,DHN:ComputingNumericallyGreens2007,LMS:ExplicitFactorizationGreens2025}.
Green's functions/fundamental solutions are also known for other specific geometries such as layered media,
wave guides in periodic media, and structured composites%
~\cite{CC:ParallelFastAlgorithm2012,LLQ:PerfectlyMatchedLayer2018,BY:WindowedGreenFunction2021,BBFT:HalfSpaceMatchingMethod2023,LYZZ:UniformFarFieldAsymptotics2024}.

For general Lipschitz domains (including boundary conditions) and, in particular, 
for operators with variable coefficients, an explicit Green’s function is typically unavailable. 
In these situations, the kernel-based Galerkin BEM \eqref{eqn:cDSP} with~\eqref{eqn:cDSP_kernel}
cannot be implemented. The kernel-free BEM formulation \eqref{eqn:DSP} removes this restriction.

\section{Finite and boundary element discretisation}%
\label{sub:FEM-based discretiation}
This section discusses convergence rates and complexity estimates of the kernel-free BEM~\eqref{eqn:DSP}
for standard discretisations by (mapped) finite and boundary elements.
\subsection{Parametric partitions of $\Omega$ and $\Sigma$}%
\label{sub:Parametric partitions}

Let $\T$ partition the Lipschitz domain $\Omega\subset\R^n$ into elements 
$T\in\T$ 
with non-overlapping interiors and $\overline\Omega=\cup\mathcal{T}$.
Each $T\in\mathcal{T}$ is the image of the (closed) $n$-dimensional 
unit simplex or unit cube $\widehat T$
under a bi-Lipschitz homeomorphism $\Phi_T:\widehat T\to T$.
The set of faces $\mathcal{F}=\mathcal{F}(\partial\Omega)\cup \mathcal{F}(\Omega)$
(images of reference faces under the element maps) is decomposed into boundary faces 
$\mathcal{F}(\partial\Omega)\coloneqq\{F\in\mathcal{F}\ :\ F\subset\partial\Omega\}$ and
interior faces $\mathcal{F}(\Omega)=\mathcal{F}\setminus\mathcal{F}(\partial\Omega)$.

We assume that $\mathcal{T}$ is shape-regular 
in the following sense,
see \Cref{rem:curved_partitions} below for connections to alternative definitions.
Denote the diameter of $T\in\T$ by $\diamT_T\coloneqq \mathrm{diam}(T)$.%
\begin{defn}[shape regularity]\label{def:shape_regular}
	The partition $\mathcal{T}$ is called 
	\begin{enumerate}[label=(\roman*)]
		\item \emph{regular} %
	if any $F\in\mathcal{F}(\Omega)$ is the face $F=T_1\cap T_2$
	of exactly two $T_1\ne T_2\in\T$ and
	\begin{align}\label{eqn:compatibility}
		\Phi_{T_2}^{-1}\circ \Phi_{T_1}|_{\widehat F_1}: \widehat F_1\to \widehat F_2
	\end{align}
	is an affine bijection between the faces 
	$\widehat F_1=\Phi_{T_1}^{-1}(F)$ and $\widehat F_2=\Phi_{T_2}^{-1}(F)$ 
	of $\widehat T$,
\item
	\emph{shape regular of order 
	$\mu\in\mathbb N_0$} if it is regular and
	any $T\in\mathcal{T}$ satisfies $\Phi_T\in C^{\mu}(\widehat T;T)$ and there exists $\Const{sr}>0$ 
	such that, for all $j=1,\dots,\max\{1,\mu\}$,
	\begin{align*}
		\|D\Phi_T^{-1}\|_{L^\infty(T)}\leq \Const{sr} \diamT_T^{-1}
		\quad\text{and}\quad
		\|D^j\Phi_T\|_{L^\infty(\widehat T)}\leq \Const{sr} \diamT_T^{j}.
	\end{align*}%
	\end{enumerate}
\end{defn}
\noindent

	Similarly, $\mathcal{P}$ denotes a shape-regular partition (in the analogous sense 
	of \Cref{def:shape_regular}) of the interface $\Sigma$
	into panels $\tau$ of diameter $\diamP_{\tau}$, which are homeomorphic to the 
	$(n-1)$-dimensional unit simplex 
	or unit cube $\widehat \tau$ under the bi-Lipschitz mapping 
	$\Phi_\tau:\widehat \tau\to \tau$.
	Note that $\mathcal{T}$ may \emph{not} resolve $\Sigma$ (and hence $\mathcal{P}$),
	i.e., $\Sigma\not\subset\cup \mathcal{F}$ is allowed.
	However, we restrict the following discussion to quasi-uniform partitions:
	There exists $\Const{qu}>0$ with
	\begin{align*}
		\diamT\coloneqq\max_{T\in\mathcal{T}}\diamT_T\leq \Const{qu} \diamT_T
		\qquad\text{and}\qquad	
		\diamP\coloneqq\max_{\tau\in\mathcal{P}}\diamP_\tau\leq \Const{qu} \diamP_{\tau}
		\qquad\text{for all }T\in\mathcal{T},\tau\in\mathcal{P}.
	\end{align*}%

	Shape regularity implies the usual pull-back scaling for Sobolev norms on elements 
	and yields uniform overlap bounds between $\mathcal{T}$ and $\mathcal{P}$.
	Recall the spacial dimension $n\geq2$.%
	\begin{lemma}[overlap control]\label{lem:localisation}
		There exists a constant $\Const{ov}>0$ that exclusively depends on 
		$n $, $\Const{sr}$, and $\Const{qu}$
		such that any $\tau\in\mathcal{P}$ 
		and $\mathcal{T}(\tau)\coloneqq\{T\in\T \ :\ T\cap \tau\ne\emptyset\}$
		satisfy
		\begin{align*}
			\max_{\tau\in\mathcal{P}}|\mathcal{T}(\tau)|\leq \Const{ov}(1+\diamP/\diamT)^n.
		\end{align*}
	\end{lemma}
	\begin{proof}
		Let $B_r(x)$ denote the open $n$-dimensional 
		ball of radius $r>0$ around $x\in\R^n$.
		The shape-regularity provides
		$0<r_0$ (that exclusively depends on $n$, $\Const{sr}$, and $\Const{qu}$) with
	\begin{align*}
		B_{r_0\diamT}(b_T)\subset T\subset B_{\diamT}(b_T)
		\qquad\text{for all }T\in\mathcal{T},
	\end{align*}
	where $b_T\coloneqq \Phi_T(\widehat b)\subset T$ is the image of the barycenter 
	(centroid) $\widehat b$ of $\widehat T$. 
	Indeed, the first inclusion follows from quasi-uniformity and the Lipschitz continuity of 
	$\Phi_T^{-1}$ that enters
	\begin{align*}
		0<\min_{\widehat x\in\partial\widehat T}\|\widehat
		x-\widehat b\|
		\leq\|D\Phi^{-1}_T\|_{L^\infty(T)}
		\min_{\widehat x\in\partial\widehat T}\|\Phi_T(\widehat
		x)-\Phi_T(\widehat b)\|
		\leq \Const{sr}/\diamT_T\min_{x\in\partial T}\|x-b_T\|.
	\end{align*}
	Any $T\in\mathcal{T}$ that intersects $B_{r d}(x)$ for $r>0$ 
	and $x\in\overline \Omega$ satisfies $T\subset B_{(1+r)\diamT}(x)$.
	Hence the number of such elements is bounded by the quotient
	$|B_{(1+r)\diamT}(0)|/|B_{r_0\diamT}(0)|=(1+r)^nr_0^{-n}$.
	This for $r=\diamP/\diamT$ and $\tau\subset B_{h}(x)$ for any 
	$x\in\tau\in\mathcal{P}$ 
	conclude the proof with $\Const{ov}\coloneqq r_0^{-n}$.
	\end{proof}

\subsection{Convergence rates of $h$-FEM/BEM}%
\label{sub:Convergence rates in }
	Let $P_m(\widehat K)$ denote polynomials up to 
	degree $m\in\mathbb{N}_0$ on the reference element
	$\widehat K\in\{\widehat T,\widehat\tau\}$
	(total degree on simplices and maximal degree on the cubes).
	Define the spaces
	\begin{align*}
		P_m(\mathcal{T})
		&\coloneqq
		\{ p \in L^\infty(\Omega) : p|_T \circ \Phi_T \in P_m(\widehat T)
	  \quad\text{for all }T\in \mathcal{T} \},\\
		P_m(\mathcal{P}) 
		&\coloneqq \{ p \in L^\infty(\Sigma) : p|_\tau \circ \Phi_\tau \in P_m(\widehat \tau)
			\quad\text{for all }\tau
		\in \mathcal{P} \}
	\end{align*}
	of piecewise mapped polynomials 
	(that depend implicitly on the mappings $\Phi_T$ and $\Phi_\tau$).
	Throughout the remaining parts of this paper, we assume that $\mathcal{T}$
	(resp.~$\mathcal{P}$) is shape regular 
	of order $p\in\mathbb{N}$ (resp.~$k\in\mathbb{N}_0$) and 
	that $\T$ resolves $\Gamma_{\mathrm{D}}$ 
	in the sense that
	\begin{align*}
	\Gamma_{\mathrm{D}}=\cup\mathcal{F}(\Gamma_{\mathrm{D}})
	\qquad\text{for }\mathcal{F}(\Gamma_{\mathrm{D}})\coloneqq\{F\in\mathcal{F}\ :\ 
	F\subset\Gamma_{\mathrm{D}}\}.
	\end{align*}
	The volume and interface partitions $\mathcal{T}$ and $\mathcal{P}$
	are then associated with the discrete spaces
	\begin{align*}
		\Vh\coloneqq S^p_{\mathrm{D}}(\mathcal{T})
		\coloneqq P_p(\mathcal{T})\cap H^1_{\Gamma_{\mathrm{D}}}(\Omega)\subset \V
		\qquad\text{and}\qquad
		\Bdh\coloneqq P_k(\mathcal{P})\subset \Bd.
	\end{align*}
	\Cref{lem:consistency} guarantees the well-posedness of the 
	BEM formulation~\eqref{eqn:DSP_intro} for these spaces under a condition on
	their approximation and Bernstein-type 
	properties~\eqref{eqn:Vh_approx}--\eqref{eqn:Bdh_inverse}.
	The involved constants
	relate to standard quasi-interpolation and inverse estimates
	recalled in the following two
	lemmata.
\begin{lemma}[{\cite{Bab:FiniteElementMethod1970,LMWZ:OptimalPrioriEstimates2010}}]
	\label{thm:V_quasi_interpolation}
	Let $m\in\{0,1\}$ and $m\leq s\leq p$.
	There exists a bounded linear operator 
	$Q:\V\to S^p_{\mathrm{D}}(\mathcal{T})$ with
	\begin{align*}
		\|v-Q v\|_{H^m(\Omega)}\leq\Const{Q} \diamT^{1+s-m}\|v\|_{H^{1+s}(\Omega)}
		\qquad\text{for all }v\in \V\cap H^{1+s}(\Omega).
	\end{align*}
	Moreover, if $\Sigma\subset\cup\mathcal{F}$ is resolved by 
	the faces $\mathcal{F}$ of $\mathcal{T}$, then
	\begin{align*}
		\|v-Q v\|_{H^m(\Omega)}\leq\Const{Q} \diamT^{1+s-m}
		\|v\|_{H^{1+s}(\Omega\setminus\Sigma)}
		\qquad\text{for all }v\in \V\cap H^{1+s}(\Omega\setminus\Sigma).
	\end{align*}
	The constant $\Const{Q}>0$ exclusively depends on $p$, $\Const{sr}$, and $\Const{qu}$.
\end{lemma}
\begin{proof}
	A careful revisit of the arguments in~\cite{LMWZ:OptimalPrioriEstimates2010}
	reveals that Theorem 3.5 therein holds for any dimension $n\geq2$ and
	Lipschitz interfaces.
	Note that Assumption 2.3 therein, 
	which enforces the correct scaling of the element maps,
	is (for $n\geq 4$) only required 
	for any $v\in H^1(T)\cap C(T)$ and $T\in\mathcal{T}$.
	In the current setting, this is implied 
	by the shape-regularity in \Cref{def:shape_regular}.ii
	and a standard scaling argument (as in%
	~\cite[Lem.~2.3]{Ber:OptimalFiniteElementInterpolation1989}, \cite{Cia:FiniteElementMethod2002}).
	Hence the statement follows from~\cite[Thm.~3.5]{LMWZ:OptimalPrioriEstimates2010} 
	(with $\delta=0$ therein).
\end{proof}

\begin{lemma}[inverse inequality~{\cite{DFG+:InverseInequalitiesNonquasiuniform2004}}]
	\label{lem:inverse_inequality}
	There exists $\Const{inv}>0$ that exclusively depends on $k$, $\Const{sr}$, and 
	$\Const{qu}$ with
	\begin{align*}
		\|\genvarBdh\|_{H^{-1/2+s}(\Sigma)}
		\leq\Const{inv} \diamP^{-s}\|\genvarBdh\|_{H^{-1/2}(\Sigma)}
		\qquad\text{for all }\genvarBdh\in P_k(\mathcal{P}), 0\leq s\leq 1/2.
	\end{align*}
\end{lemma}
\begin{proof}
	The case $s=1/2$ is a straightforward modification of%
	~\cite[Thm.~4.6]{DFG+:InverseInequalitiesNonquasiuniform2004}.
	Interpolation with the trivial case $s=0$ concludes the proof; 
	further details are omitted.
\end{proof}
We hesitate to include here a precise discussion on the geometric assumptions that allow
higher-order approximation results also for 
\emph{nearly resolved} interfaces in \textbf{\Cref{thm:B}} and 
in \Cref{thm:V_quasi_interpolation}; 
for details we refer~\cite{LMWZ:OptimalPrioriEstimates2010}
and the following remark that
precedes the proof of \textbf{\Cref{thm:B}} from the introduction with positive $\sigma_{\mathrm{reg}}>0$.
\begin{remark}[$\delta$-resolved interface]\label{rem:delta_resolve_interface}
	The higher-order approximation in \Cref{thm:V_quasi_interpolation} 
	holds also for 
	interfaces $\Sigma\not\subset\cup\mathcal{F}$ that are perturbations of exactly
	resolved interfaces%
	~\cite[Thm.~3.5]{LMWZ:OptimalPrioriEstimates2010}.
	Intuitively, the interface $\Sigma$ is $\delta$-resolved by $\mathcal{T}$ if
	there exists $\Sigma_{\diamT}\subset\cup \mathcal{F}$ with distance 
	to $\Sigma$ bounded by some sufficiently small $\delta\ll \diamT$,
	which enters the approximation error bound.
	In this case, the convergence result in \textbf{\Cref{thm:B}}.ii becomes
	\begin{align*}
		\|\solvarBd-\solvarBdh\|_{H^{-1/2}(\Sigma)}
		\leq \Const{rate}(\diamT^{\min\{s, p+1\}} + \delta +  \diamP^{\min\{s,k+1\}})
			\|\datavarB\|_{H^{1/2+s}(\Sigma)}.
	\end{align*}
\end{remark}

\begin{proof}[Proof of~\textbf{\Cref{thm:B}}]
	Observe for any $0\leq s\leq s_*<1/2$ with $0<s_*\leq \sigma_{\mathrm{reg}}$
	that $\Bdh\subset \Bd^{s}=H^{-1/2+s}(\Sigma)$
	and $\W^s\subset H^{1+s}(\Omega\setminus\Sigma)\cap H^1(\Omega)=H^{1+s}(\Omega)$ 
	hold.
	\Cref{thm:V_quasi_interpolation,lem:inverse_inequality}
	provide~\eqref{eqn:Vh_approx}--\eqref{eqn:Bdh_inverse} with 
	$\eta_{s}(\Vh)\leq \Const{Q}\diamT^s$ and $\Const{inv}\diamP^s\leq \beta_s(\Bdh)$.
	Hence \Cref{thm:quasi_best} establishes the well-posedness and 
	the quasi-optimality~\eqref{eqn:Sh_quasi_best}
	of~\eqref{eqn:DSP_intro} provided that
	\begin{align}\label{eqn:cmp_def}
		\diamT^s\leq \Const{cmp}(s)\diamP^s
		\qquad\text{with}\quad\Const{cmp}(s)\coloneqq \Const{inv}\const{\opL}/(2\Const{Q}\Const{con}(s))
	\end{align}
	for some $s\in(0,s_*]$ and $\Const{con}(s)$ from \Cref{lem:consistency}.
	This holds in particular if
	$\diamT\leq \Const{cmp}\diamP$ for 
	$\Const{cmp}\coloneqq\inf_{s\in (0,s_*]}(\Const{cmp}(s))^{1/s}>0$,
	which is assumed hereafter.

	Given $\datavarB\in\B^s=H^{1/2+s}(\Sigma)$
	with $0\leq s\leq \sigma_{\mathrm{reg}}$,
	the continuous regularity shift
	of \Cref{thm:regularity of solutions} 
	and the standard approximation properties 
	(from a scaling-argument) of $\Bdh=P_k(\mathcal{P})$ 
	and \Cref{thm:V_quasi_interpolation} 
	(with $H^{1+s}(\Omega\setminus\Sigma)\cap H^1(\Omega)=H^{1+s}(\Omega)$ if $s<1/2$)
	reveal
	\begin{align*}
		\min_{\genvarBdh\in \Bdh}\|\solvarBd-\genvarBdh\|_{H^{-1/2}(\Sigma)}
		&\lesssim \diamP^{\min\{s,1+k\}}\|\datavarB\|_{H^{1/2+s}(\Sigma)},\\
		\min_{\genvarVh\in \Vh}\|\opSL\solvarBd-\genvarVh\|_{H^1(\Omega)}
		&\lesssim \left\{\begin{array}{ll}
				\diamT^s&\text{if }s<1/2,\\
				\diamT^{\min\{s,1+p\}}
					&\text{if }s>1/2\text{ and }\Sigma\subset\cup \mathcal{F}
		\end{array}\right\}
	\times \|\datavarB\|_{H^{1/2+s}(\Sigma)}.
	\end{align*}
	This and the quasi-optimality~\eqref{eqn:Sh_quasi_best} conclude the proof of 
	\textbf{\Cref{thm:B}}.
\end{proof}

\begin{remark}[curved partitions]\label{rem:curved_partitions}
	Many different formulations of curved/parametric discretisations
	exist in the literature, along with strategies for a practical construction
	of the element maps $\Phi_T$.
	Common approaches include (iso-)parametric element maps
	derived by interpolation of the boundaries/interfaces%
	~\cite[Chap.~3--4]{Cia:FiniteElementMethod2002},~\cite[Chap.~8]{SS:BoundaryElementMethods2011}
	and element maps obtained from the restriction of given global
	parametrisations of (parts of) the physical domain over a simpler 
	reference domain as discussed, e.g., 
	in~\cite[Sec.~5]{MS:ConvergenceAnalysisFinite2010} and \cite{OR:TraceFiniteElement2017}.
	Both of these strategies lead to shape-regular partitions in the sense of
	\Cref{def:shape_regular} and are included 
	in the unified presentation of this section. 
\end{remark}

\subsection{Data-sparse single-layer approximation}%
\label{sub:Algebraic structure and sparsity}

Recall the algebraic structure from \Cref{sub:Algebraic formulation} and observe
from the shape-regularity that the 
dimensions of the discrete spaces $\Vh=S^p_{\mathrm{D}}(\mathcal{T})$ and
$\Bdh=P_k(\mathcal{P})$ satisfy
\begin{align*}
	\NVh=\operatorname{dim}\Vh\in\mathcal{O}(\diamT^{-n})
	\qquad\text{and}\qquad
	\NBdh=\operatorname{dim}\Bdh\in\mathcal{O}(\diamP^{1-n}).
\end{align*}
Under the mesh compatibility condition $\diamT\lesssim \diamP$
from~\eqref{eqn:rel_mesh_size}, the inverse $\Vh$-stiffness matrix
$N_\diamT=L_\diamT^{-1}\in \mathbb C^{\NVh\times\NVh}$ is substantially larger than
the single-layer matrix (cf.~\eqref{eqn:Vhef})
\begin{align}\label{eqn:single_layer_matrix_5}
	V_{\diamT}^{\diamP} = (R_{\diamT}^{\diamP})^{H}N_{\diamT}R_{\diamT}^\diamP\in\mathbb{K}^{\NBdh\times \NBdh}
\end{align}
and the remaining parts of
this section discuss the efficient computation of~\eqref{eqn:single_layer_matrix_5}.

Abbreviate $[J]\coloneqq \{1,\dots,J\}$ for any $J\in\mathbb{N}_0$.
Choose local bases $\{\genvarVh[j]\}_{j=1}^{\NVh}$ and $\{\genvarBdh[\ell]\}_{\ell=1}^{\NBdh}$ 
of $\Vh$ and $\Bdh$ such that,
for every $j\in[\NVh]$ and $\ell\in[\NBdh]$, there are
$T_j\in\mathcal{T}$ and $\tau_\ell\in\mathcal{P}$ with
\begin{align}\label{eqn:proxy_elem}
	\operatorname{supp}\genvarVh[j]
	\subset\overline{\Omega_{T_j}}
	\coloneqq\bigcup\{K\in\mathcal{T}\ :\ T_j\cap K\ne\emptyset\}
	\qquad\text{and}\qquad
	\operatorname{supp}\genvarBdh[\ell]\subset \tau_\ell.
\end{align}
This localisation reflects that the single-layer operator is supported on
$\Sigma$ and results in a sparse transfer matrix
$R_{\diamT}^{\diamP}\in\mathbb{K}^{\NVh\times \NBdh}$.
Recall $\Const{ov}$ from \Cref{lem:localisation}.
	\begin{lemma}[sparsity]\label{lem:sparsity}
		For any $\ell\in[\NBdh]$, the $\ell$-th column of $R_\diamT^\diamP$
		has at most $|\mathcal{N}_{\Sigma}(\ell)| \leq \Const{R}$ 
		nonzero entries that correspond to the row indices
		\begin{align*}
			\mathcal{N}_{\Sigma}(\ell)
			\coloneqq\{j\in[\NVh]\ :\ 
			\operatorname{supp}\genvarVh[j]\cap 
			\operatorname{supp}\genvarBdh[\ell]\ne\emptyset\}
			\qquad\text{with}\qquad\Const{R}\coloneqq\Const{ov}(1+\diamP/\diamT)^n.
		\end{align*}
		In particular, $R_\diamT^\diamP$ is sparse 
		with at most $\operatorname{nnz}R_\diamT^\diamP\leq
		\Const{R}\,\NBdh$ nonzero entries.
	\end{lemma}
	\begin{proof}
		Combine $|\mathcal{N}_{\Sigma}(\ell)|\leq |\mathcal{T}(\tau_{\ell})|$ 
		with $\tau_\ell\in\mathcal{P}$ from~\eqref{rem:delta_resolve_interface}
		and \Cref{lem:localisation}.
	\end{proof}
	
	Each nonzero entry of $R_{\diamT}^{\diamP}$ is an $L^2$-pairing of locally
	supported (mapped) piecewise polynomial basis functions over an intersection $\tau\cap T$ of
	$\tau\in\mathcal P$ and $T\in\mathcal T(\tau)$.
	By \Cref{lem:sparsity}, all nonzero entries are contained in the sparse submatrix
	\begin{align}\label{eqn:R_block_def}
		R_{\diamT,\Sigma}^{\diamP}
		\coloneqq R_{\diamT}^{\diamP}\big|_{\mathcal N_\Sigma\times[\NBdh]}
		\in\mathbb C^{|\mathcal N_\Sigma|\times\NBdh},
		\qquad
		\mathcal N_\Sigma\coloneqq\bigcup_{\ell\in[\NBdh]}\mathcal N_\Sigma(\ell)
	\end{align}
	with $|\mathcal N_\Sigma|\leq \Const{R}\NBdh$ and
	$\operatorname{nnz}R_{\diamT,\Sigma}^{\diamP}\leq \Const{R}\NBdh$.
	Hence the assembly cost of $R_{\diamT,\Sigma}^{\diamP}$ is proportional to
	$\Const{R}\NBdh$ (up to computing the intersection integral, e.g., by a fixed quadrature rule).

	For the complexity estimate of~\eqref{eqn:single_layer_matrix_5},
	a standard sparse--dense matrix multiplication (SpMM)
	model is adopted: if $R$ is sparse and $M$ is dense with $m\in\mathbb{N}$ columns, then the
	products $MR$ and $R^{H}M$ can be evaluated in
	$\mathcal{O}(m\,\operatorname{nnz}R)$ operations.
	
	\begin{theorem}[complexity]\label{thm:complexity}
		Let $R_{\diamT,\Sigma}^\diamP$ and $N_{\diamT}$ be given and assume the above SpMM model.
		The single-layer matrix $V_\diamT^\diamP$ can be computed in 
		$\mathcal{O}((1+\diamP/\diamT)^{2n}\NBdh^2)$ operations.
		In particular, if $\diamP\lesssim\diamT$, then the computational complexity 
		is $\mathcal{O}(\NBdh^2)$.
	\end{theorem}

	\begin{proof}
		The sparsity structure of $R_\diamT^\diamP$ from \Cref{lem:sparsity} and~\eqref{eqn:R_block_def}
		simplify~\eqref{eqn:single_layer_matrix_5} to
		\begin{align}\label{eqn:single_layer_matrix_simplify}
			V_{\diamT}^{\diamP} = 
			(R_{\diamT,\Sigma}^{\diamP})^{H}N_{\diamT,\Sigma}R_{\diamT,\Sigma}^\diamP
			\qquad\text{with}\qquad 
			N_{\diamT,\Sigma} 
			\coloneqq N_{\diamT}|_{\mathcal{N}_{\Sigma}\times\mathcal{N}_{\Sigma}}.
		\end{align}
		Since $\max\{\operatorname{nnz}R_{\diamT,\Sigma}^\diamP, |\mathcal{N}_{\Sigma}|\}\leq\Const{R}\NBdh$,
		the SpMM model computes $M\coloneqq N_{\diamT,\Sigma}R_{\diamT,\Sigma}^\diamP$ and 
		$V_\diamT^\diamP=(R_{\diamT,\Sigma}^{\diamP})^HM$ in
		$\mathcal{O}(\Const{R}^2\NBdh^2)$ operations.
		(Observe that the dense submatrix $N_{\diamT,\Sigma}$ can be materialised from $N_\diamT$ in 
		$\mathcal{O}(\Const{R}^2\NBdh^2)$.)
		This and $\Const{R}\approx(1+\diamP/\diamT)^n$ conclude the proof.
	\end{proof}

	Hence \textbf{\Cref{thm:C}} from the introduction holds: If $N_{\diamT,\Sigma}$ is available 
	(e.g., from a preprocessing step) and $\diamP\lesssim\diamT$, then the assembly cost of 
	$V_\diamT^\diamP$ is quadratic $\mathcal{O}(\NBdh^2)$ in the boundary degrees of freedom, 
	matching the cost of forming the dense Galerkin matrix in classical BEM.
	Together with \textbf{\Cref{thm:B}}, this motivates the natural scaling $\diamT\approx\diamP$.

\begin{remark}[data-sparse approximation]\label{rem:data_sparse}
	In classical BEM, the quadratic $\mathcal{O}(\NBdh^2)$ complexity for assembling
	and storing dense system matrices is typically avoided by data-sparse
	representations, 
	which support fast arithmetic
	such as $\mathcal{H}$- and $\mathcal{H}^2$-matrices%
	~\cite{Hac:HierarchicalMatricesAlgorithms2015,
	BH:ExistenceHmatrixApproximants2003,
	Beb:EfficientInversionGalerkin2004,
	Bor:ApproximationSolutionOperators2010}.
	The factorisation~\eqref{eqn:single_layer_matrix_simplify} exhibits the same
	structure for the discrete single-layer operator $V_{\diamT}^{\diamP}$:
	the transfer matrix $R_{\diamT,\Sigma}^{\diamP}$ is sparse, while all dense
	information is contained in the restricted inverse stiffness matrix
	$N_{\diamT,\Sigma}$.
	Consequently, any data-sparse approximation $\widetilde N_{\diamT}$ of
	$N_{\diamT}=L_{\diamT}^{-1}$ induces a data-sparse
	realisation 
	\[
		\widetilde V_{\diamT}^{\diamP}
		\coloneqq(R_{\diamT,\Sigma}^{\diamP})^{H}\,
		\widetilde N_{\diamT,\Sigma}\,R_{\diamT,\Sigma}^{\diamP},
		\qquad
		\widetilde N_{\diamT,\Sigma}
		\coloneqq\widetilde N_{\diamT}|_{\mathcal{N}_{\Sigma}\times\mathcal{N}_{\Sigma}}.
	\]
	The substitution of $N_{\diamT}$ by $ \widetilde N_{\diamT}$ enters the error analysis in
	\Cref{cor:quasi-best approximation} only through the perturbation term
	\eqref{eqn:Nh_approx_error}, which can be reduced to any prescribed tolerance in applications.
	For elliptic finite element discretisations, accurate $\mathcal H$-matrix approximants
	$\widetilde N_{\diamT}$ are available with (stretched-)exponential error decay in
	the (local) block rank~\cite{FMP:HmatrixApproximability2015}.
\end{remark}

\section*{Acknowledgements}
The authors gratefully acknowledge the financial support by 
the Swiss National Science Foundation under grant no.\ 2000-1-240043.

\bibliographystyle{alphaabbr}
\bibliography{Bibliography}
\end{document}